\documentclass[leqno,10pt]{amsart}

\usepackage{amsfonts, amsmath, amssymb,amscd,indentfirst, graphicx}

\numberwithin{equation}{section}
\newtheorem{definition}{Definition}
\newtheorem{proposition}{Proposition}
\newtheorem{lemma}{Lemma}
\newtheorem{theorem}{Theorem}
\newtheorem{remark}{Remark}
\newtheorem{corollary}{Corollary}

\theoremstyle{definition}

\begin{document}

\newcommand{\riem}{(M^m, \langle \, , \, \rangle)}
\newcommand{\Hess}{\mathrm{Hess}\, }
\newcommand{\hess}{\mathrm{hess}\, }
\newcommand{\ess}{\mathrm{ess}}
\newcommand{\longra}{\longrightarrow}
\newcommand{\eps}{\varepsilon}
\newcommand{\vol}{\mathrm{vol}}
\newcommand{\di}{\mathrm{d}}
\newcommand{\R}{\mathbb R}
\newcommand{\C}{\mathbb C}
\newcommand{\Z}{\mathbb Z}
\newcommand{\N}{\mathbb N}
\newcommand{\HH}{\mathbb H}
\newcommand{\Sph}{\mathbb S}
\newcommand{\metric}{\langle \, , \, \rangle}
\newcommand{\metricN}{( \, , \, )}
\newcommand{\lip}{\mathrm{Lip}}
\newcommand{\loc}{\mathrm{loc}}
\newcommand{\diver}{\mathrm{div}}
\newcommand{\disp}{\displaystyle}
\newcommand{\rad}{\mathrm{rad}}
\newcommand{\Ricc}{\mathrm{Ricc}}
\newcommand{\sn}{\mathrm{sn}}
\newcommand{\cn}{\mathrm{cn}}
\newcommand{\ink}{\mathrm{in}}
\newcommand{\hol}{\mathrm{H\ddot{o}l}}
\newcommand{\capac}{\mathrm{cap}}
\newcommand{\ricc}{\operatorname{Ricc}}
\newcommand{\supp}{\operatorname{supp}}
\newcommand{\sgn}{\operatorname{sgn}}
\newcommand{\F}{\mathcal{F}}
\newcommand{\cut}{\mathrm{cut}}
\newcommand{\rk}{\mathrm{rk}}
\newcommand{\crit}{\mathrm{crit}}
\newcommand{\diam}{\mathrm{diam}}
\newcommand{\haus}{\mathcal{H}}
\newcommand{\gr}{\mathcal{G}}
\newcommand{\Bo}{\mathbb{B}}
\newcommand{\ra}{\rightarrow}
\newcommand{\dist}{\mathrm{dist}}
\newcommand{\II}{\mathrm{II}}
\newcommand{\DD}{\mathbb{D}}
\newcommand{\NN}{\mathbb{N}}
\title{Density and spectrum of minimal submanifolds in space forms}


%    author one information
\author[B. P. Lima]{B. Pessoa Lima}
\address{Dep. Matem\'atica, UFPI, Campus Ministro Petr\^onio Portela, 64049-550 Teresina - PI}
\curraddr{}
\email{barnabe@ufpi.edu.br}
\thanks{}

\author[L. Mari]{L. Mari}
\address{Dep. Matem\'atica, UFC, Campus do Pici - Bloco 914, 60.455-760, Fortaleza - CE }
\curraddr{}
\email{mari@mat.ufc.br}
\thanks{The second author is supported by the grant PRONEX - N\'ucleo de An\'alise Geom\'etrica e Aplicac\~oes 
Processo nº PR2-0054-00009.01.00/11}

\author[J. F. Montenegro]{J. Fabio Montenegro}
\address{Dep. Matem\'atica, UFC, Campus do Pici - Bloco 914, 60.455-760, Fortaleza - CE }
\curraddr{}
\email{fabio@mat.ufc.br}
\thanks{The third author is partially supported by CNPq}

%    author two information
\author[F. B. Vieira]{F. de Brito Vieira}
\address{Dep. Matem\'atica, UFC, Campus do Pici - Bloco 914, 60.455-760, Fortaleza - CE }
\curraddr{}
\email{cianev@hotmail.com}
\thanks{}

%    \subjclass is required.
% 58J50  View Publications (2000-now) Spectral problems; spectral geometry; scattering theory
% 53C21  View Publications (1980-now) Methods of Riemannian geometry, including PDE methods; curvature restrictions

\date{\today}

\dedicatory{}

%    "Communicated by" -- provide editor's name; required.
\commby{}

\maketitle

\begin{abstract}
Let $\varphi: M^m \ra N^n$ be a minimal, proper immersion in an ambient space suitably close to a space form $\mathbb{N}^n_k$ of curvature $-k\le 0$. In this paper, we are interested in the relation between the density function $\Theta(r)$ of $M$ and the spectrum of its Laplace-Beltrami operator. In particular, we prove that if $\Theta(r)$ has subexponential growth (when $k<0$) or sub-polynomial growth ($k=0$) along a sequence, then the spectrum of $M^m$ is the same as that of the space form $\mathbb{N}^m_k$. Notably, the result applies to Anderson's (smooth) solutions of Plateau's problem at infinity on the hyperbolic space, independently of their boundary regularity. We also give a simple condition on the second fundamental form that ensures $M$ to have finite density. In particular, we show that minimal submanifolds with finite total curvature in the hyperbolic space also have finite density.
%\subclass{58J50 \and 58C21 \and 35P15}
\end{abstract}

\section{Introduction}
\label{intro}

Let $M^m$ be a minimal, properly immersed submanifold in a complete ambient space $N^n$. In the present paper, we are interested in the case when $N$ is close, in a sense made precise below, to a space form $\NN_k^n$ of curvature $-k\le 0$. In particular, our focus is the study of the spectrum of the Laplace Beltrami operator $-\Delta$ on $M$ and its relationship with the density at infinity of $M$, that is, the limit as $r \ra +\infty$ of the (monotone) quantity
\begin{equation}\label{def_densityinfty_hyp}
\Theta(r) \doteq \frac{\vol(M \cap B_r)}{V_k(r)},
\end{equation}
where $B_r$ indicates a geodesic ball of radius $r$ in $N^n$ and $V_k(r)$ is the volume of a geodesic ball of radius $r$ in $\NN^m_k$. Hereafter, we will say that $M$ has finite density if 
$$
\Theta(+\infty) \doteq \lim_{r \ra +\infty} \Theta(r) <+\infty. 
$$
To properly put our results into perspective, we briefly recall few facts about the spectrum of the Laplacian on a geodesically complete manifold. It is known by works of P. Chernoff \cite{chernoff} and R.S. Strichartz \cite{strichartz} that $-\Delta$ on a complete manifold is essentially self-adjoint on the domain $C^\infty_c(M)$, and thus it admits a unique self-adjoint extension, which we still call $-\Delta$. Since $-\Delta$ is positive and self-adjoint, its spectrum is the set of $\lambda \geq0$ such that
$\Delta +\lambda I$ does not have bounded inverse. Sometimes we say spectrum of $M$ rather than spectrum of $-\Delta$  and we denote it by $\sigma(M)$.
The well-known Weyl's characterization for the spectrum of a self-adjoint operator in a Hilbert space implies the following
\begin{lemma}\cite[Lemma 4.1.2]{Davies}\label{lem_weyl}
\label{lem3}
A number $\lambda \in \mathbb{R}$ lies in $\sigma(M)$ if and only if there exists a sequence of nonzero functions
$u_j\in \mathrm{Dom}(-\Delta)$ such that
\begin{equation}\label{L2norm} 
\|\Delta u_j+ \lambda u_j\|_{2} = o\big( \|u_j\|_2\big) \qquad \text{as } \, j \ra +\infty.
\end{equation}
\end{lemma}
In the literature, characterizations of the whole $\sigma(M)$ are known only in few special cases. Among them, the Euclidean space, for which $\sigma(\R^m) = [0,\infty)$, and the hyperbolic space $\HH^m_k$, for which 
\begin{equation}\label{spectrum_Hm}
\sigma(\HH^m_k) = \left[ \frac{(m-1)^2k}{4}, +\infty \right)\!. 
\end{equation}
The approach to guarantee that $\sigma(M) = [c, +\infty)$, for some $c \ge 0$, usually splits into two parts. The first one is to show that $\inf \sigma(M) \ge c$ via, for instance, the Laplacian comparison theorem from below (\cite{mckean}, \cite{GPB}), and the second one is to produce a sequence like in lemma \ref{lem3} for each $\lambda>c$. This step is accomplished by considering radial functions of compact support, and, at least in the first results on the topic like the one in \cite{donnelly}, uses the comparison theorems on both sides for $\Delta \rho$, $\rho$ being the distance from a fixed origin $o \in M$. Therefore, the method needs both a pinching on the sectional curvature and the smoothness of $\rho$, that is, that $o$ is a pole of $M$ (see \cite{donnelly}, \cite{escobarfreire},\cite{jli} and Corollary 2.17 in \cite{BMR_memoirs}), which is a severe topological restriction. Since then, various efforts were made to weaken both the curvature and the topological assumptions. We briefly overview some of the main achievements.\par
In \cite{Kumura}, Kumura observed that to perform the second step (and just for it) it is enough that there exists a relatively compact, mean convex, smooth open set $\Omega$ with the property that the normal exponential map realizes a global diffeomorphism $\partial \Omega \times \R_0^+ \ra M \backslash \Omega$. Conditions of this kind seem, however, unavoidable for his techniques to work. On the other hand, in \cite{Kumura_2} the author drastically weakened the curvature requirements needed to establish Step 2, by replacing the two-sided pinching on the sectional curvature with a combination of a lower bound on a suitably weighted volume and an $L^p$-bound on the Ricci curvature. 

Regarding the need for a pole, major recent improvements have been made in a series of papers (\cite{sturm}, \cite{wang}, \cite{Lu}, \cite{charalambouslu}): their guiding idea was to replace the $L^2$-norm in relation \eqref{L2norm} with the $L^1$-norm, which via a trick in \cite{wang}, \cite{Lu} enables to use smoothed distance functions to construct sequences as in Lemma \ref{lem_weyl}. Building on deep function-theoretic results due to Sturm \cite{sturm} and Charalambous-Lu \cite{charalambouslu}, in \cite{wang}, \cite{Lu} the authors proved that $\sigma(M) =[0,\infty )$ when
\begin{equation}\label{iporicci}
\liminf_{\rho(x) \ra +\infty} \Ricc_x = 0
\end{equation}
in the sense of quadratic forms, without any topological assumption. This remarkable result improves on \cite{jli} and \cite{escobarfreire} (see also Corollary 2.17 in \cite{BMR_memoirs}), where $M$ was assumed to have a pole. Further refinements of \eqref{iporicci} have been given in \cite{charalambouslu}. However, when \eqref{iporicci} does not hold, the situation is more delicate and is still the subject of an active area of research. In this respect, we also quote the general function-theoretic criteria developed by H. Donnelly \cite{donnelly_exha}, and K.D. Elworthy and F-Y. Wang \cite{elworthywang} to ensure that a half-line belongs to the spectrum of $M$.\par
The main concern in this paper is to achieve, in the above-mentioned setting of minimal submanifolds $\varphi : M \ra N$, a characterization of the whole $\sigma(M)$ free from curvature or topological conditions on $M$ (in this respect, observe that the completeness of $M$ follows from that of $N$ and the properness of $\varphi$). It is known by \cite{Cheung} and \cite{GPB} that for a minimal immersion $\varphi : M^m \ra \NN^n_k$ the fundamental tone of $M$, $\inf \sigma(M)$, is at least that of $\NN^m_k$, i.e.,
\begin{equation}\label{infspec_intro}
\inf \sigma(M) \ge \frac{(m-1)^2k}{4}.
\end{equation}
Moreover, as a corollary of \cite{Kumura} and \cite{BJM}, \cite{BC}, if the second fundamental form $\II$ satisfies the decay estimate 
\begin{equation}\label{intro_decayII}
\begin{array}{ll}
\disp \lim_{\rho(x) \ra +\infty} \rho(x)|\II(x)| =0 & \quad \text{if } \, k=0 \\[0.3cm]
\disp \lim_{\rho(x) \ra +\infty} |\II(x)| =0 & \quad \text{if } \, k>0 \\[0.2cm]  
\end{array}
\end{equation}
($\rho(x)$ being the intrinsic distance with respect to some fixed origin $o \in M$), then $M$ has the same spectrum that a totally geodesic submanifold $\NN^m_k \subset \NN^n_k$, that is, 
\begin{equation}\label{spectrum_McomeHm}
\sigma(M) = \left[ \frac{(m-1)^2k}{4}, +\infty \right)\!. 
\end{equation}
According to \cite{Anderson_prep}, \cite{filho}, \eqref{intro_decayII} is ensured when $M$ has finite total curvature, that is, when
\begin{equation}\label{finite_total}
\int_M |\II|^m < +\infty.
\end{equation}
\begin{remark}
\emph{A characterization of the essential spectrum, similar to \eqref{spectrum_McomeHm}, also holds for submanifolds of the hyperbolic space $\HH^n_k$ with constant (normalized) mean curvature $H<\sqrt{k}$. There, condition \eqref{finite_total} is replaced by the finiteness of the $L^m$-norm of the traceless second fundamental form. For deepening, see \cite{castillon}.
}
\end{remark}
Condition \eqref{intro_decayII} is a quite binding requirement for \eqref{spectrum_McomeHm} to hold, since it needs a pointwise control of the second fundamental form, and the search for more manageable conditions has been at the heart of the present paper. Here, we identify a suitable growth on the density function $\Theta(r)$ \emph{along a sequence} as a natural candidate to replace it, see \eqref{bellissima}. As a very special case, \eqref{spectrum_McomeHm} holds when $M$ has finite density. It might be interesting that just a volume growth condition along a sequence could control the whole spectrum of $M$; for this to happen, the minimality condition enters in a crucial and subtle way. 

Regarding the relation between \eqref{finite_total} and the finiteness of $\Theta(+\infty)$, we remark that their interplay has been investigated in depth for minimal submanifolds of $\R^n$, but the case of $\HH^n_k$ seems to be partly unexplored. In the next section, we will briefly discuss the state of the art, to the best of our knowledge. As a corollary of Theorem \ref{teo_finitedens} below, we will show the following  
\begin{corollary}\label{cor_densitycurvature}
Let $M^m$ be a minimal properly immersed submanifold in $\HH^n_k $. If $M$ has finite total curvature, then $\Theta(+\infty)<+\infty$.
\end{corollary}
As far as we know, this result was previously known just in dimension $m=2$ via a Chern-Osserman type inequality, see the next section for further details.\par
We now come to our results, beginning with defining the ambient spaces which we are interested in: these are manifolds with a pole, whose radial sectional curvature is suitably pinched to that of the model $\NN^n_k$.
\begin{definition}\label{def_closeHn}
Let $N^n$ possess a pole $\bar o$ and denote with $\bar \rho$ the distance function from $\bar o$. Assume that the radial sectional curvature $\bar K_\rad$ of $N$, that is, the sectional curvature restricted to planes $\pi$ containing $\bar\nabla \bar \rho$, satisfies
\begin{equation}\label{pinchsectio}
- G\big( \bar \rho(x) \big) \le \bar K_\rad(\pi_x) \le -k \le 0 \qquad \forall \, x \in N \backslash \{\bar o\},
\end{equation}
for some $G \in C^0(\R^+_0)$. We say that
\begin{itemize}
\item[$(i)$] \emph{$N$ has a pointwise (respectively, integral) pinching to $\R^n$ if $k=0$ and 
$$
sG(s) \ra 0 \ \text{ as } \, s \ra +\infty \qquad \big(\textrm{respectively, $\, sG(s) \in L^1(+\infty)$}\big);
$$
}
\item[$(ii)$] \emph{$N$ has a pointwise (respectively, integral) pinching to $\HH^n_k$ if $k>0$ and 
$$
G(s)-k \ra 0 \ \text{ as } \, s \ra +\infty \qquad \big(\textrm{respectively, $\, G(s)-k \in L^1(+\infty)$}\big).
$$
}
\end{itemize}

\end{definition}

Hereafter, given an ambient manifold $N$ with a pole $\bar o$, the density function $\Theta(r)$ will always be computed by taking extrinsic balls centered at $\bar o$.\par 
Our main achievements are the following two theorems. The first one characterizes $\sigma(M)$ when the density of $M$ grows subexponentially (respectively, sub-polynomially) along a sequence. Condition \eqref{bellissima} below is very much in the spirit of a classical growth requirement due to R. Brooks \cite{brooks} and Y. Higuchi \cite{higuchi} to bound from above the infimum of the essential spectrum of $-\Delta$. However, we stress that our Theorem \ref{teo_spectrum} seems to be the first result in the literature characterizing the whole spectrum of $M$ under just a mild volume assumption.

\begin{theorem}\label{teo_spectrum}
Let $\varphi : M^m \ra N^n$ be a minimal properly immersed submanifold, and suppose that $N$ has a pointwise or an integral pinching to a space form. If either
\begin{equation}\label{bellissima}
\begin{array}{ll}
\text{$N$ is pinched to $\HH^n_k$, and} & \qquad \disp \liminf_{s \ra +\infty} \frac{\log \Theta(s)}{s} = 0, \quad \text{or } \\[0.4cm]
\text{$N$ is pinched to $\R^n$, and} & \qquad \disp \liminf_{s \ra +\infty} \frac{\log \Theta(s)}{\log s} = 0.
\end{array}
\end{equation}
then
\begin{equation}\label{wholespectrum}
\sigma(M) = \left[ \frac{(m-1)^2k}{4}, +\infty\right)\!.
\end{equation}
\end{theorem}

The above theorem is well suited for minimal submanifolds constructed via Geometric Measure Theory since, typically, their existence is guaranteed by controlling the density function $\Theta(r)$. As an important example, Theorem \ref{teo_spectrum} applies to all solutions of Plateau's problem at infinity $M^m \ra \HH^n_k$ constructed in \cite{Anderson}, provided that they are smooth. Indeed, because of their construction, $\Theta(+\infty)<+\infty$ (see \cite{Anderson}, part [A] at p. 485) and they are proper (it can also be deduced as a consequence of $\Theta(+\infty)<+\infty$, see Remark \ref{rem_proper}). By standard regularity theory, smoothness of $M^m$ is automatic if $m \le 6$.

\begin{corollary}\label{cor_plateau}
Let $\Sigma \subset \partial_\infty \HH^n_k$ be a closed, integral $(m-1)$ current in the boundary at infinity of $\HH^n_k$ such that, for some neighbourhood $U\subset \HH^ n_k$ of $\supp(\Sigma)$, $\Sigma$ does not bound in $U$, and let $M^m \hookrightarrow \HH^n_k$ be the solution of Plateau's problem at infinity constructed in \cite{Anderson} for $\Sigma$. If $M$ is smooth, then \eqref{wholespectrum} holds.
\end{corollary}

An interesting fact of Corollary \ref{cor_plateau} is that $M$ is \emph{not} required to be regular up to $\partial_\infty \HH^n_k$, in particular it might have infinite total curvature. In this respect, we observe that if $M$ be $C^2$ up to $\partial_\infty \HH^n_k$, then $M$ would have finite total curvature (Lemma \ref{prop_mazet} in Appendix 1). By deep regularity results, this is the case if, for instance, $M^m \ra \HH^{m+1}_k$ is a smooth hypersurface that solves Plateau's problem for $\Sigma$, and $\Sigma$ is a $C^{2,\alpha}$ (for $\alpha>0$), embedded compact hypersurface of $\partial_\infty \HH^n_k$. See Appendix 1 for details.\par
The spectrum of solutions of Plateau's problems has also been considered in \cite{PLM} for minimal surfaces in $\R^3$. In this respect, it is interesting to compare Corollary \ref{cor_plateau} with $(3)$ of Corollary 2.6 therein.

\begin{remark}
\emph{The solution $M$ of Plateau's problem in \cite{Anderson} is constructed as a weak limit of a sequence $M_j$ of minimizing currents for suitable boundaries $\Sigma_j$ converging to $\Sigma$. and property $\Theta(+\infty)<+\infty$ is a consequence of a uniform upper bound for the mass of a sequence $M_j$ (part [A], p. 485 in \cite{Anderson}). Such a bound is achieved because of the way the boundaries $\Sigma_j$ are constructed, in particular, since they are all sections of the same cone. One might wonder whether $\Theta(+\infty)<+\infty$, or at least the subexponential growth in \eqref{bellissima}, is satisfied by all solutions of Plateau's problem. In this respect, we just make this simple observation: in the hypersurface case $n=m+1$, if $M \cap B^{m+1}_r$ is volume-minimizing then clearly 
$$
\Theta(r) = \frac{\vol(M \cap B^{m+1}_r)}{V_k(r)} \le \frac{\vol( \partial B_r^{m+1} \subset \HH_k^{m+1})}{V_k(r)} = c_k \frac{\sinh^m(\sqrt{k}r)}{V_k(r)},
$$
but this last expression diverges exponentially fast as $r \ra +\infty$ (differently from its Euclidean analogous, which is finite). This might suggest that a general solution of Plateau's problem does not automatically satisfies $\Theta(+\infty)<+\infty$, and maybe not even \eqref{bellissima}.
}
\end{remark}

In our second result we focus on the particular case when $\Theta(+\infty) <+\infty$, and we give a sufficient condition for its validity in terms of the decay of the second fundamental form. Towards this aim, we shall restrict to ambient spaces with an integral pinching.

\begin{theorem}\label{teo_finitedens}
Let $\varphi : M^m \ra N^n$ be a minimal immersion, and suppose that $N$ has an integral pinching to a space form. Denote with $\rho(x)$ the intrinsic distance from some reference origin $o \in M$. Assume that there exist $c>0$ and $\alpha>1$ such that the second fundamental form satisfies, for $\rho(x) >>1$,
\begin{equation}\label{approaching_hyp}
\begin{array}{ll}
\disp |\II(x)|^2 \le \frac{c}{\rho(x) \log^{\alpha}\rho(x)} & \qquad \text{if $N$ is pinched to $\HH^n_k$;} \\[0.4cm]
\disp |\II(x)|^2 \le \frac{c}{\rho(x)^2 \log^{\alpha}\rho(x)} & \qquad \text{if $N$ is pinched to $\R^n$.}
\end{array}
\end{equation}
Then, $\varphi$ is proper, $M$ is diffeomorphic to the interior of a compact manifold with boundary, and $\Theta(+\infty)<+\infty$.
\end{theorem}
The assertions that $\varphi$ be proper and $M$ have finite topology is well-known under assumptions even weaker than \eqref{approaching_hyp} and not necessarily requiring the minimality, see for instance \cite{BJM}, \cite{BC}. Former results are due to \cite{Anderson_prep} ($N= \R^n$) and \cite{filho}, \cite{castillon} ($N=\HH^n_k$). Here, our original contribution is to show that $M$ has finite density. Because of a result in \cite{filho}, \cite{PigolaVeronelli}, if $\varphi : M \ra \HH^n_k$ has finite total curvature then $|II(x)| = o(\rho(x)^{-1})$ as $\rho(x) \ra +\infty$. Hence, \eqref{approaching_hyp} is met and Corollary \ref{cor_densitycurvature} follows at once.\par 
We briefly describe the strategy of the proof of Theorem \ref{teo_spectrum}. In view of \eqref{infspec_intro}, it is enough to show that each $\lambda > (m-1)^2 k/4$ lies in $\sigma(M)$. To this end, we follow an approach inspired by a general result due to K.D. Elworthy and F-Y. Wang \cite{elworthywang}. However, Elworthy-Wang's theorem is not sufficient to conclude, and we need to considerably refine the criterion in order to fit in the present setting. To construct the sequence as in Lemma \ref{lem_weyl}, a key step is to couple the volume growth requirement \eqref{bellissima} with a sharpened form of the monotonicity formula for minimal submanifolds, which improves on the classical ones in \cite{simon}, \cite{Anderson}. Indeed, in Proposition \ref{prop_monotonicity} we describe three monotone quantities other than $\Theta(s)$, that might be useful beyond the purpose of the present paper. For example, in the very recent \cite{GimenoMarkvosen} the authors discovered and used some of the relations in Proposition \ref{prop_monotonicity} to show interesting comparison results for the capacity and the first eigenvalue of minimal submanifolds.  
\subsection{Finite density and finite total curvature in $\R^n$ and $\HH^n$}
The first attempt to extend the classical theory of finite total curvature surfaces in $\R^n$ (see \cite{Osserman_book}, \cite{JorgeMeeks}, \cite{ChernOsserman_1}, \cite{ChernOsserman_2}) to the higher-dimensional case is due to M.T. Anderson. In \cite{Anderson_prep}, the author drew from \eqref{finite_total} a number of topological and geometric consequences, and here we focus on those useful to highlight the relationship between total curvature and density. First, he showed that \eqref{finite_total} implies the decay
\begin{equation}\label{decay_Rn}
\lim_{\rho(x) \ra +\infty} \rho(x)|\II(x)| =0,
\end{equation} 
where $\rho(x)$ is the intrinsic distance from a fixed origin, and as a consequence $M$ is proper, the extrinsic distance function $r$ has no critical points outside some compact set and $|\nabla r| \ra 1$ as $r$ diverges, so by Morse theory $M$ is diffeomorphic to the interior of a compact manifold with boundary. Moreover, he proved that $M$ has finite density  via a higher-dimensional extension of the Chern-Osserman identity \cite{ChernOsserman_1}, \cite{ChernOsserman_2}, namely the following relation linking the Euler characteristic $\chi(M)$ and the Pfaffian form $\Omega$ (\cite{Anderson_prep}, Theorem 4.1):
\begin{equation}\label{chern_oss_Rm}
\chi(M) = \int_M \Omega + \lim_{r \ra +\infty} \frac{\vol(M \cap \partial B_r)}{V_0'(r)}. 
\end{equation}
Observe that, since $|\nabla r| \ra 1$, by coarea's formula the limit in the right hand-side coincides with $\Theta(+\infty)$. We underline that property $\Theta(+\infty)<+\infty$ plays a fundamental role to apply the machinery of manifold convergence to get information on the limit structure of the ends of $M$ (\cite{Anderson_prep}, \cite{ShenZhu}, \cite{Tysk}). For instance, $\Theta(+\infty)$ is related to the number $\mathcal{E}(M)$ of ends of $M$: if we denote with $V_1, \ldots, V_{\mathcal{E}(M)}$ the (finitely many) ends of $M$, \eqref{finite_total} implies for $m \ge 3$ the identities
\begin{equation}\label{numberfinals}
\Theta(+\infty) = \sum_{i=1}^{\mathcal{E}(M)} \lim_{r \ra +\infty} \frac{\vol(V_i \cap \partial B_r)}{V_0'(r)}  \equiv \mathcal{E}(M),
\end{equation}
and thus $M$ is totally geodesic provided that it has only one end and finite total curvature (\cite{Anderson_prep}, Thm 5.1 and its proof). Further information on the mutual relationship between the finiteness of the total curvature and $\Theta(+\infty)<+\infty$ can be deduced under the additional requirement that $M$ is stable or it has finite stability index. For example, by work of J. Tysk \cite{Tysk}, if $M^m$ has finite index and $m \le 6$, then 
\begin{equation}\label{equiindex}
\Theta(+\infty)<+\infty \qquad \text{if and only if} \qquad \int_M |\II|^m <+\infty.
\end{equation}
\begin{remark}
\emph{Indeed, the main result in \cite{Tysk} states that, when $\Theta(+\infty)<+\infty$ and $m \le 6$, $M$ has finite index if and only if it has finite total curvature. However, since the finite total curvature condition alone implies both that $M$ has finite index and $\Theta(+\infty)<+\infty$ (in any dimension\footnote{As said, finite total curvature implies $\Theta(+\infty)<+\infty$ by \eqref{chern_oss_Rm}, while the finiteness of the index can be seen as an application of the generalized Cwikel-Lieb-Rozembljum inequality (see \cite{liyau}) to the stability operator $L = -\Delta -|\II|^2$, recalling that a minimal submanifold $M^m \ra \R^n$ satisfies a Sobolev inequality. We refer to \cite{PRS} for deepening.}), the characterization in \eqref{equiindex} is equivalent to Tysk's theorem. We underline that it is still a deep open problem whether or not, for $m \ge 3$, stability or finite index alone implies the finiteness of the density at infinity.
}
\end{remark}
%Note that, according to a celebrated result of R. Schoen, L. Simon and S.T. Yau \cite{SchoenSimonYau}, this would prove the generalized Bernstein theorem for stable minimal hypersurfaces of dimension $m \le 5$.\footnote{By the generalized Bernstein theorem in $\R^{m+1}$ we mean the statement  that the only stable minimal hypersurfaces of $\R^{m+1}$ are geodesic planes. The statement is false if $m \ge 8$ in view of the celebrated counterexample in \cite{bdgg}, it is true for $m=2$ and still unknown if $3 \le m \le 7$. Note that, in dimension $m =7$, Simons'cone \cite{simons} provides a counterexample to the conjecture if we allow $M$ to be singular, but the Bernstein problem might still be true for smooth minimal hypersurfaces. 
%}
%
Since then, efforts were made to investigate analogous properties for minimal submanifolds of finite total curvature immersed in $\HH^n_k$. There, some aspects show strong analogy with the $\R^n$ case, while others are strikingly different. For instance, minimal immersions $\varphi : M^m \ra \HH^n_k$ with finite total curvature enjoy the same decay property \eqref{decay_Rn} with respect to the intrinsic distance $\rho(x)$ (\cite{filho}, see also \cite{PigolaVeronelli}), which is enough to deduce that they are properly immersed  and diffeomorphic to the interior of a compact manifold with boundary. Moreover, Anderson \cite{Anderson} proved the monotonicity of $\Theta(r)$ in \eqref{def_densityinfty_hyp}. In order to show (among other things) that complete, finite total curvature surfaces $M^2 \hookrightarrow \HH^n$ have finite density, in \cite{Chen}, \cite{ChenCheng} the authors obtained the following Chern-Osserman type inequality:
\begin{equation}\label{chern_osserman_hyp}
\chi(M) \ge - \frac{1}{4\pi} \int_M |\II|^2 + \Theta(+\infty),
\end{equation}
see also \cite{GimenoPalmer_2}. However, in the higher dimensional case we found no analogous of \eqref{chern_oss_Rm}, \eqref{chern_osserman_hyp} in the literature, and adapting the proof of \eqref{chern_oss_Rm} to the hyperbolic ambient space seems to be subtler than what we expected. In fact, an \emph{equality} like \eqref{chern_oss_Rm} is not even possible to obtain, since there exist minimal submanifolds of $\HH^n_k$ with finite density but whose density at infinity depends on the chosen reference origin \cite{Gimeno_private}. We point out that, on the contrary, inequality \eqref{chern_osserman_hyp} holds for each choice of the reference origin in $\R^n$. This motivated the different route that we follow to prove Theorem \ref{teo_finitedens} and Corollary \ref{cor_densitycurvature}.
Among the results in \cite{Anderson_prep} that could not admit a corresponding one in $\HH_k^n$, in view of the solvability of Plateau's problem at infinity on $\HH_k^n$ we stress that a relation like \eqref{numberfinals} cannot hold for each minimal submanifold of $\HH^n_k$ with finite total curvature. Indeed, there exist a wealth of properly immersed minimal submanifolds in $\HH^n_k$ with finite total curvature and one end: for example, referring to the upper half-space model, the graphical solution of Plateau's problem for $\Sigma^{m-1} \subset \partial_\infty \HH^n_k$ being the boundary of a convex set (constructed at the end of \cite{Anderson}) has finite total curvature, as follows from Lemma \ref{prop_mazet} and the regularity results recalled in Appendix 1. It shall be observed, however, that when $\II$ decays sufficiently fast at infinity with respect to the extrinsic distance function $r(x)$: 
\begin{equation}\label{decay_Hn}
\lim_{r(x) \ra +\infty} e^{2 \sqrt{k}r(x)}|\II(x)| =0,
\end{equation} 
then the inequality $\Theta(+\infty) \le \mathcal{E}(M)$ still holds for minimal hypersurfaces in $\HH^n_k$ as shown in \cite{GimPal}, and in particular $M$ is totally geodesic provided that it has only one end, as first observed in \cite{KasueSugahara}, \cite{KasueSugahara_2}. We remark that there exists an infinite family of complete minimal cylinders $\varphi_\lambda : \Sph^1 \times \R \ra \HH^3$ whose second fundamental form $\II_\lambda$ decays exactly of order $\exp\{-2r(x)\}$, see \cite{Mori}.

\section{Preliminaries}
Let $\varphi : (M^m, \metric) \rightarrow (N^n, \metricN)$ be an isometric immersion of a complete $m$-dimensional Riemannian manifold $M$ into an ambient manifold $N$ of dimension $n$ and possessing a pole $\bar o$. We denote with $\nabla, \Hess, \Delta$ the connection, the Riemannian Hessian and the Laplace-Beltrami operator on $M$, while quantities related to $N$ will be marked with a bar. For instance, $\bar\nabla, \overline{\dist}, \overline{\Hess}$ will identify the connection, the distance function and the Hessian in $N$. Let $\bar \rho(x)= \overline{\dist}(x,\bar o)$ be the distance function from $\bar o$. Geodesic balls in $N$ of radius $R$ and center $y$ will be denoted with $B_R^N(y)$. Moreover, set  
\begin{equation}\label{def_r}
r \ : \ M \ra \R, \qquad r(x) = \bar\rho\big(\varphi(x)\big),
\end{equation}
for the extrinsic distance from $\bar o$. We will indicate with $\Gamma_{\! s}$ the extrinsic geodesic spheres restricted to $M$: 
$\Gamma_{\! s} \doteq \{x\in M;\;r(x)=s\}$. Fix a base point $o \in M$. In what follows, we shall also consider the intrinsic distance function $\rho(x) = \dist(x,o)$ from a reference origin $o \in M$. 

\subsection{Target spaces}
Hereafter, we consider an ambient space $N$ possessing a pole $\bar o$ and, setting $\bar \rho(x) \doteq \dist(x, \bar o)$, we assume that \eqref{pinchsectio} is met for some $k \ge 0$ and some $G \in C^0(\R^+_0)$. Let $\sn_k(t)$ be the solution of
\begin{equation}
\left\{\begin{array}{l}
\sn_k'' - k\, \sn_k = 0 \quad \text{on } \, \R^+, \\[0.1cm]
\sn_k(0)=0, \quad \sn_k'(0)=1, 
\end{array}\right.
\end{equation}
that is
\begin{equation}\label{def_snk}
\sn_k(t) = \left\{ \begin{array}{ll} t & \quad \text{if } \, k=0, \\[0.1cm]
 \sinh(\sqrt{k}t)/\sqrt{k} & \quad \text{if } \, k>0. 
\end{array}\right.
\end{equation}
Observe that $\R^n$ and $\HH^n_k$ can be written as the differentiable manifold $\R^n$ equipped with the metric given, in polar geodesic coordinates $(\rho, \theta) \in \R^+ \times \Sph^{n-1}$ centered at some origin, by 
$$
\di s^2_k = \di \rho^2 + \sn^2_k(\rho)\,\di \theta^2, 
$$
$\di \theta^2$ being the metric on the unit sphere $\Sph^{n-1}$.\\
We also consider the model $M^n_g$ associated with the lower bound $-G$ for $\bar K_\rad$, that is, we let $g \in C^2(\R^+_0)$ be the solution of  
\begin{equation}\label{def_g}
\left\{\begin{array}{l}
g'' - Gg = 0 \quad \text{on } \, \R^+, \\[0.1cm]
g(0)=0, \quad g'(0)=1, 
\end{array}\right.
\end{equation}
and we define $M^n_g$ as being $(\R^n, \di s^2_g)$ with the $C^2$-metric $\di s_g^2 = \di \rho^2 + g^2(\rho) \di \theta^2$ in polar coordinates. Condition \eqref{pinchsectio} and the Hessian comparison theorem (Theorem 2.3 in \cite{PRS}, or Theorem 1.15 in \cite{BMR_memoirs}) imply
\begin{equation}\label{hessiancomp}
\frac{\sn_k'(\bar\rho)}{\sn_k(\bar\rho)} \Big( \metricN - \di \bar \rho \otimes\di \bar\rho\Big) \le  \overline{\Hess}(\bar\rho) \le \frac{g'(\bar \rho)}{g(\bar \rho)}\Big( \metricN - \di \bar \rho \otimes\di \bar\rho\Big).
\end{equation}
The next proposition investigates the ODE properties that follow from the assumptions of pointwise or integral pinching.
\begin{proposition}\label{prop_pinching}
Let $N^n$ satisfy \eqref{pinchsectio}, and let $\sn_k,g$ be solutions of \eqref{def_snk}, \eqref{def_g}. Define 
\begin{equation}\label{rel_useful}
\zeta(s) \doteq \frac{g'(s)}{g(s)} - \frac{\sn_k'(s)}{\sn_k(s)}.
\end{equation}
Then, $\zeta(0^+)=0$, $\zeta \ge 0$ on $\R^+$. Moreover, 
\begin{itemize}
\item[$(i)$] If $N$ has a pointwise pinching to $\HH^n_k$ or $\R^n$, then $\zeta(s) \ra 0$ as $s \ra +\infty$.
%$$
%\zeta(s) \frac{\sn_k(s)}{\sn_k'(s)} \ra 0 \qquad \text{as } \, s \ra +\infty.
%$$
\item[$(ii)$] If $N$ has an integral pinching to $\HH^n_k$ or $\R^n$, then $g/\sn_k \ra C$ as $s \ra +\infty$ for some $C \in \R^+$, and 
\begin{equation}\label{intebounds_zeta}
\zeta(s) \in L^1(\R^+), \qquad \zeta(s) \frac{\sn_k(s)}{\sn_k'(s)} \ra 0 \ \text{ as } \, s \ra +\infty.
\end{equation}
\end{itemize}
\end{proposition}

\begin{proof}
The non-negativity of $\zeta$, which in particular implies that $g/\sn_k$ is non-decreasing, follows from $G\ge k$ via Sturm comparison, and $\zeta(0^+)=0$ depends on the asymptotic relations $\sn_k'/\sn_k = s^{-1} + o(1)$ and $g'/g = s^{-1} + o(1)$ as $s \ra 0^+$, which directly follow from the ODEs satisfied by $\sn_k$ and $g$. To show $(i)$, differentiating $\zeta$ we get 
\begin{equation}\label{relazeta}
\zeta'(s) = R(s) - \zeta(s) B(s),
\end{equation}
where $R(s) \doteq G(s) - k$ and $\disp B(s) \doteq \frac{g'(s)}{g(s)} + \frac{\sn_k'(s)}{\sn_k(s)}$.
Thus, integrating on $[1,s]$, we can rewrite $\zeta$ as follows:
\begin{equation}\label{rewritezeta}
\zeta(s) = \zeta(1)e^{ - \int_1^sB} + e^{ - \int_1^sB}\int_1^s R(\sigma)e^{\int_1^\sigma B}\di \sigma
\end{equation}
Using that $B \not \in L^1([1,+\infty))$, and applying de l'Hopital's theorem, we infer 
$$
\lim_{s \ra +\infty} \zeta(s) = \lim_{s \ra +\infty}\frac{R(s)}{B(s)} \le \lim_{s \ra +\infty} \frac{\sn_k(s)[G(s)-k]}{\sn_k'(s)}.
$$
In our pointwise pinching assumptions on $G(s)$, for both $k=0$ and $k>0$ the last limit is zero, hence $\zeta(s) \ra 0$ as $s$ diverges. To show $(ii)$, suppose that $N$ has an integral pinching to $\HH^n_k$ or to $\R^n$. We first observe that the boundedness of $g/\sn_k$ on $\R^+$ equivalent to the property $\zeta \in L^1(+\infty)$, as it follows from 
\begin{equation}
\log \frac{g(s)}{\sn_k(s)} = \int_0^{s} \frac{\di}{\di \sigma}\log\left(\frac{g(\sigma)}{\sn_k(\sigma)}\right)\di s = \int_0^{s}\zeta
\end{equation}
(we used that $(g/\sn_k)(0^+) =1$). The boundedness of $g/\sn_k$ is the content of Corollary 4 and Remark 16 in \cite{BMR_Yamabe},
%\footnote{The reader be warned that $k$ and $G$ here are, respectively, $G$ and $\bar G$ in Corollary 4 of \cite{BMR_Yamabe}.}
but we prefer here to present a direct proof. Integrating \eqref{rewritezeta} on $[1,s]$ and using Fubini's theorem, the monotonicity of $g/\sn_k$ and the expression of $B$ we obtain
\begin{equation}\label{equa_bonita}
\begin{array}{lcl}
\disp \int_1^s \zeta & = & \disp \zeta(1) \int_1^s \frac{g(1)\sn_k(1)}{g(\sigma)\sn_k(\sigma)}\,\di \sigma + 
                                    \int_1^s e^{ - \int_1^\sigma B}\int_1^\sigma R(\tau)e^{\int_1^\tau B}\di \tau \, \di \sigma \\[0.5cm]
& \le & \disp \zeta(1)\sn_k(1)^2 \int_1^s \frac{\di \sigma}{\sn_k^2(\sigma)} + \int_1^s \left[\int_\tau^s e^{ - \int_1^\sigma B}R(\tau)e^{\int_1^\tau B}\di \sigma\right] \di \tau \\[0.5cm]
%& \le & \disp C + \int_1^s R(\tau)e^{\int_1^\tau B}\left[\int_\tau^s e^{ - \int_1^\sigma B}\di \sigma\right] \di \tau \\[0.5cm] 
& \le & \disp C + \int_1^s R(\tau)g(\tau)\sn_k(\tau)\left[\int_\tau^s \frac{\di \sigma}{g(\sigma)\sn_k(\sigma)}\right] \di \tau \\[0.5cm] 
& \le & \disp C + \int_1^s R(\tau)g(\tau)\sn_k(\tau)\left[\int_\tau^{+\infty} \frac{\di \sigma}{g(\sigma)\sn_k(\sigma)}\right] \di \tau \\[0.5cm] 
\end{array}
\end{equation}
for some $C>0$, where we have used that $\sn_k^{-2}, g^{-1}\sn_k^{-1} \in L^1(+\infty)$. Next, since $g\,\sn_k/\sn_k^2$ is non-decreasing, Proposition 3.12 in \cite{BMR_memoirs} ensures the validity of the following inequality: 
$$
g(\tau)\sn_k(\tau)\left[\int_\tau^{+\infty} \frac{\di \sigma}{g(\sigma)\sn_k(\sigma)}\right] \le \sn_k^2(\tau)\left[\int_\tau^{+\infty} \frac{\di \sigma}{\sn_k^2(\sigma)}\right].
$$
It is easy to show that the last expression is bounded if $k>0$, and diverges at the order of $\tau$ if $k=0$. In other words, it can be bounded by $C_1\sn_k/\sn_k'$ on $[1,+\infty)$, for some large $C_1>0$. Therefore, by \eqref{equa_bonita}
$$
\disp \int_1^s \zeta \le \disp C + C_1\int_1^s R(\tau)\frac{\sn_k(\tau)}{\sn_k'(\tau)} \di \tau = \disp C + C_1\int_1^s \big[G(\tau)-k\big]\frac{\sn_k(\tau)}{\sn_k'(\tau)} \di \tau.
$$
In our integral pinching assumptions, both for $k=0$ and for $k>0$ it holds $(G-k)\sn_k/\sn_k' \in L^1(+\infty)$, and thus $\zeta \in L^1(+\infty)$. Next, we use \eqref{relazeta} and the non-negativity of $\zeta,B$ to obtain
$$
\begin{array}{lcl}
\disp \left( \frac{\zeta(s)\sn_k(s)}{\sn_k'(s)}\right)' & = & \disp \big[G(s)-k - \zeta(s)B(s)\big]\frac{\sn_k(s)}{\sn_k'(s)} + \zeta(s)\left[ 1 - k \left(\frac{\sn_k(s)}{\sn_k'(s)}\right)^2\right] \\[0.5cm]
& \le & \disp \frac{\big[G(s)-k\big]\sn_k(s)}{\sn_k'(s)} + \zeta(s) \in L^1(+\infty), 
\end{array}
$$
hence $\zeta \sn_k/\sn_k' \in L^\infty(\R^+)$ by integrating. This implies that the function $B$ in \eqref{relazeta} satisfies $B \le C\sn_k'/\sn_k$ for some constant $C>0$. Therefore, from \eqref{relazeta} we get $\zeta' \ge -\zeta B \ge - C \zeta \sn_k'/\sn_k$. Integrating on $[s,t]$ and using the monotonicity of $\sn_k'/\sn_k$ we obtain
$$
- C\frac{\sn_k'(s)}{\sn_k(s)}\int_s^t \zeta \le \zeta(t) -\zeta(s).
$$
Since $\zeta \in L^1(\R^+)$, we can choose a divergent sequence $\{t_j\}$ such that $\zeta(t_j) \ra 0$ as $j \ra +\infty$. Setting $t=t_j$ into the above inequality and taking limits we deduce
$$
\zeta(s) \le C\frac{\sn_k'(s)}{\sn_k(s)}\int_s^{+\infty} \zeta, 
$$
thus letting $s \ra +\infty$ we get the second relation in \eqref{intebounds_zeta}.
\end{proof}

\subsection{A transversality lemma}
This subsection is devoted to an estimate of the measure of the critical set 
$$
S_{t,s} = \Big\{ x \in M \ : \ t \le r(x) \le s, \ |\nabla r(x)|=0 \Big\},
$$
with the purpose of justifying some coarea's formulas for integrals over extrinsic annuli. We begin with the next
\begin{lemma}\label{lem_coarea}
Let $\varphi : M^m \ra N^n$ be an isometric immersion, and let $r(x) = \overline{\dist}(\varphi(x), \bar o)$ be the extrinsic distance function from $\bar o \in N$. Denote with $\Gamma_{\!\sigma} \doteq \{x\in M;\; r(x)=\sigma\}$. Then, for each $f \in L^1(\{t \le r \le s\})$, 
\begin{equation}\label{coarea_strong}
\int_{\{t \le r \le s\}} f \,\di x = \int_{S_{t,s}}\!f \,\di x + \int_t^s \left[ \int_{\Gamma_{\!\sigma}} \frac{f}{|\nabla r|}\right] \di \sigma.
\end{equation}
In particular, if 
\begin{equation}\label{assu_Hm}
\vol(S_{t,s})=0, 
\end{equation}
then 
\begin{equation}\label{coarea_strong2}
\int_{\{t \le r \le s\}} f \,\di x = \int_t^s \left[ \int_{\Gamma_{\!\sigma}} \frac{f}{|\nabla r|}\right] \di \sigma.
\end{equation}
\end{lemma}

\begin{proof}
We prove \eqref{coarea_strong} for $f \ge 0$, and the general case follows by considering the positive and negative part of $f$. By the coarea's formula, we know that for each $g \in L^1(\{t \le r \le s\})$,
\begin{equation}\label{classical_coarea}
\int_{\{t \le r \le s\}} g|\nabla r| \,\di x = \int_t^s \left[ \int_{\Gamma_{\!\sigma}} g\right] \di \sigma.
\end{equation}
Fix $j$ and consider $A_j = \{ |\nabla r|> 1/j\}$ and the function $$g = f1_{A_j}/|\nabla r| \in L^1(\{t \le r \le s\}).$$ Applying \eqref{classical_coarea}, letting $j \ra +\infty$ and using the monotone convergence theorem we deduce
\begin{equation}\label{limit}
\int_{\{t \le r \le s\} \backslash S_{t,s}} f \,\di x = \int_t^s \left[ \int_{\Gamma_{\!\sigma} \backslash S_{t,s}} \frac{f}{|\nabla r|}\right] \di \sigma = \int_t^s \left[ \int_{\Gamma_{\!\sigma}} \frac{f}{|\nabla r|}\right] \di \sigma, 
\end{equation}
where the last equality follows since $\Gamma_\sigma \cap S_{t,s} = \emptyset$ for a.e. $\sigma \in [t,s]$, in view of Sard's theorem. Formula \eqref{coarea_strong} follows at once.

%However, if $\{\Omega_j\}$ is an increasing exhaustion of $\{t < r < s\}$ by compact subsets, again by \eqref{classical_coarea} with $g = 1_{S_{t,s}}1_{\Omega_j}$ and letting $j \ra +\infty$, 
%$$
%0 = \int_{\{t \le r \le s\}}1_{S_{t,s}}|\nabla r| \,\di x= \int_t^s \haus^{m-1}(\Gamma_{\!\sigma} \cap S_{t,s}) \di \sigma,
%$$
%where $\haus^{m-1}$ is the $(m-1)$-dimensional Hausdorff measure. Therefore, 
%$$\haus^{m-1}(\Gamma_{\!\sigma} \cap S_{t,s}) = 0$$ for a.e. $\sigma \in [t,s]$, which implies that \eqref{limit} becomes
%\begin{equation}\label{limit}
%\int_{\{t \le r \le s\} \backslash S_{t,s}} f \,\di x = \int_t^s \left[ \int_{\Gamma_{\!\sigma}} \frac{f}{|\nabla r|}\right] \di \sigma,
%\end{equation}
%from which \eqref{coarea_strong} immediately follows.
\end{proof}

Let now $N$ possess a pole $\bar o$ and satisfy \eqref{pinchsectio}, and consider a minimal immersion $\varphi: M\ra N$. Since, by the Hessian comparison theorem, geodesic spheres in $N$ centered at $\bar o$ are positively curved, it is reasonable to expect that the ``transversality" condition \eqref{assu_Hm} holds. This is the content of the next 

\begin{proposition}\label{important!}
Let $\varphi: M^{m} \ra N^n$ be a minimal immersion, where $N$ possesses a pole $\bar o$ and satisfies \eqref{pinchsectio}.
Then, 
\begin{equation}
\vol(S_{0,+\infty}) = 0.
\end{equation}

\end{proposition}

\begin{proof}
Suppose by contradiction that $\vol(S_{0,+\infty})>0$. By Stampacchia and Rademacher's theorems, 
\begin{equation}\label{stampa}
\nabla |\nabla r|(x) = 0 \qquad \text{for a.e. } \, x \in S_{0,+\infty}.
\end{equation}
Pick one such $x$ and a local Darboux frame $\{e_i\}, \{e_\alpha\}$, $1 \le i \le m$, $m+1 \le \alpha \le n$ around $x$, that is, $\{e_i\}$ is a local orthonormal frame for $TM$ and $\{e_\alpha\}$ is a local orthonormal frame for the normal bundle $TM^\perp$. Since $\nabla r(x)=0$, then $\bar \nabla \bar \rho(x) \in T_xM^\perp$. Up to rotating $\{e_\alpha\}$, we can suppose that $\bar \nabla \bar \rho(x) = e_n(x)$. Fix $i$ and consider a unit speed geodesics $\gamma: (-\eps,\eps) \ra M$ such that $\gamma(0)=x$, $\dot \gamma(0)=e_i$. Identify $\gamma$ with its image $\varphi \circ \gamma$ in $N$. By Taylor's formula and \eqref{stampa},
$$
|\nabla r|(\gamma(t)) = o(t) \qquad \text{as } \, t \ra 0^+.
$$
Using that $|\nabla r| = \sqrt{ 1- \sum_\alpha( \bar \nabla \bar \rho, e_\alpha)^2}$, we deduce
\begin{equation}\label{buono_0}
1 - \sum_\alpha( \bar \nabla \bar \rho, e_\alpha)_{\gamma(t)}^2 = o(t^2).
\end{equation}
Since $\bar \nabla \bar \rho(x) = e_n(x)$, we deduce from \eqref{buono} that also
\begin{equation}\label{buono}
u(t) \doteq 1 - ( \bar \nabla \bar \rho, e_n)_{\gamma(t)}^2 = o(t^2),
\end{equation}
thus $\dot u(0) = \ddot u(0)=0$. Computing, 
$$
\begin{array}{lcl}
\dot u(t) & = & 2(\bar \nabla \bar \rho, e_n) \left[ (\bar \nabla_{\dot \gamma} \bar \nabla \bar \rho, e_n) + ( \bar \nabla \bar \rho, \bar \nabla_{\dot \gamma} e_n)\right] \\[0.2cm]
\ddot u(t) & = & 2 \left[ (\bar \nabla_{\dot \gamma} \bar \nabla \bar \rho, e_n) + ( \bar \nabla \bar \rho, \bar \nabla_{\dot \gamma} e_n)\right]^2 \\[0.2cm]
& & + 2 (\bar \nabla \bar \rho, e_n)\left[\disp ( \bar \nabla_{\dot \gamma} \bar \nabla_{\dot \gamma} \bar \nabla \bar \rho, e_n) + 2 (\bar \nabla_{\dot \gamma} \bar \nabla \bar \rho, \bar \nabla_{\dot \gamma} e_n) +  (\bar \nabla \bar \rho, \bar \nabla_{\dot \gamma} \bar \nabla_{\dot \gamma} e_n)\right].
\end{array}
$$
Evaluating at $t=0$ we deduce
$$
0 = \ddot u(0)/2 = \disp ( \bar \nabla_{e_i} \bar \nabla_{e_i} \bar \nabla \bar \rho, \bar \nabla \bar \rho) + 2 (\bar \nabla_{e_i} \bar \nabla \bar \rho, \bar \nabla_{e_i} e_n) +  (e_n, \bar \nabla_{e_i} \bar \nabla_{e_i} e_n).
$$
Differentiating twice $1 = |e_n|^2 = |\bar \nabla \bar \rho|^2$ along $e_i$ we deduce the identities $(e_n, \bar \nabla_{e_i} \bar \nabla_{e_i} e_n) = -|\bar \nabla_{e_i} e_n|^2$ and $( \bar \nabla_{e_i} \bar \nabla_{e_i} \bar \nabla \bar \rho, \bar \nabla \bar \rho) = -|\bar \nabla_{e_i} \bar \nabla \bar \rho|^2$, hence
$$
0 = \ddot u(0)/2 = \disp -|\bar \nabla_{e_i}\bar \nabla \bar \rho|^2 + 2 (\bar \nabla_{e_i} \bar \nabla \bar \rho, \bar \nabla_{e_i} e_n) - |\bar \nabla_{e_i} e_n|^2 = - \big| \bar \nabla_{e_i} \bar \nabla \bar \rho - \bar \nabla_{e_i} e_n\big|^2,
$$																												which implies $\bar \nabla_{e_i} \bar \nabla \bar \rho = \bar \nabla_{e_i} e_n$. Therefore, at $x$, 
$$
(\II(e_i,e_i),e_n) = - (\bar \nabla_{e_i} e_n, e_i) = - (\bar \nabla_{e_i} \bar \nabla \bar \rho, e_i) = \overline{\Hess}(\bar \rho)(e_i,e_i).																
$$
Tracing with respect to $i$, using that $M$ is minimal and \eqref{hessiancomp} we conclude that
$$
0 \ge \frac{\sn_k'(r(x))}{\sn_k(r(x))} (m- |\nabla r(x)|^2) = m\frac{\sn_k'(r(x))}{\sn_k(r(x))} >0,
$$
a contradiction.
\end{proof}

\section{Monotonicity formulae and conditions equivalent to $\Theta(+\infty)<+\infty$}\label{sec_mono}

Our first step is to improve the classical monotonicity formula for $\Theta(r)$, that can be found in \cite{simon} (for $N=\R^n$) and \cite{Anderson} (for $N=\HH^n_k$). For $k \ge 0$, let $v_k, V_k$ denote the volume function, respectively, of geodesic spheres and balls in the space form of sectional curvature $-k$ and dimension $m$, i.e.,
\begin{equation}\label{def_vk}
v_k(s) = \omega_{m-1}\sn_k(s)^{m-1}, \qquad V_k(s) = \int_0^s v_k(\sigma) \di \sigma,
\end{equation}
where $\omega_{m-1}$ is the volume of the unit sphere $\Sph^{m-1}$. Although we shall not use all the four monotone quantities in \eqref{monotones} below, nevertheless they have independent interest, and for this reason we state the result in its full strength. We define the \emph{flux} $J(s)$ of $\nabla r$ over the extrinsic sphere $\Gamma_s$:
\begin{equation}\label{def_J}
J(s) \doteq \frac{1}{v_k(s)} \int_{\Gamma_s} |\nabla r|.
\end{equation}

\begin{proposition}[The monotonicity formulae]\label{prop_monotonicity}
Suppose that $N$ has a pole $\bar o$ and satisfies \eqref{pinchsectio}, and let $\varphi : M^m \ra N^n$ be a proper minimal immersion. Then, the functions
\begin{equation}\label{monotones}
\Theta(s), \qquad \frac{1}{V_k(s)}\int_{\{0 \le r \le s\}}|\nabla r|^2
\end{equation}
are absolutely continuous and monotone non-decreasing. Moreover, $J(s)$ coincides, on an open set of full measure, with the absolutely continuous function
$$
\bar J(s) \doteq \frac{1}{v_k(s)}\int_{\{r \le s\}} \Delta r
$$
and $\bar J(s)$, $V_k(s)\big[ \bar J(s)-\Theta(s)\big]$ are non-decreasing. In particular, $J(s) \ge \Theta(s)$ a.e. on $\R^+$.
\end{proposition}

\begin{remark}\label{rem_tkachev}
\emph{To the best of our knowledge, the monotonicity of $J(s)$ (aside from its differentiability properties) has first been shown, in the Euclidean setting, in a paper by V. Tkachev \cite{tkachev}.
}
\end{remark}

\begin{proof}
We first observe that, in view of Lemma \ref{lem_coarea} and Proposition \ref{important!} applied with $f = \Delta r$, 
\begin{equation}\label{usefullll}
v_k(s)\bar J(s) \doteq \int_{\{r \le s\}} \Delta r \equiv \int_0^s \left[\int_{\Gamma_\sigma}\frac{\Delta r}{|\nabla r|}\right]\di \sigma
\end{equation}
is absolutely continuous, and by the divergence theorem it coincides with $v_k(s)J(s)$ whenever $s$ is a regular value of $r$. Consider
\begin{equation}\label{def_f}
f(s) = \int_0^s \frac{V_k(\sigma)}{v_k(\sigma)} \di \sigma = \int_0^s \frac{1}{v_k(\sigma)} \left[\int_0^\sigma v_k(\tau) \di \tau \right]\di \sigma 
\end{equation}
which is a $C^2$ solution of  
$$
f'' + (m-1) \frac{\sn_k'}{\sn_k} f' = 1 \quad \text{on } \, \R^+, \quad f(0)=0, \quad f'(0)=0,
$$
and define $\psi(x) = f(r(x)) \in C^2(M)$. Let $\{e_i\}$ be a local orthonormal frame on $M$. Since $\varphi$ is minimal, by the chain rule and the lower bound in the Hessian comparison theorem \ref{hessiancomp}
\begin{equation}\label{basica}
\Delta r = \sum_{j=1}^m \overline{\Hess}(\bar \rho)\big(\di \varphi(e_j), \di \varphi(e_j)\big) \ge \frac{\sn_k'(r)}{\sn_k(r)} \big(m -|\nabla r|^2\big).
\end{equation}
We then compute
\begin{equation}\label{deltaf}
\begin{array}{lcl}
\disp \Delta \psi & = & \disp f''|\nabla r|^2 + f'\Delta r  \ge \disp f''|\nabla r|^2 + f'\frac{\sn_k'}{\sn_k}(m-|\nabla r|^2) \\[0.3cm]
& = & \disp \disp 1 + \left(1-|\nabla r|^2\right)\left(f'(r)\frac{\sn_k'(r)}{\sn_k(r)}- f''(r)\right).
\end{array}
\end{equation}
It is not hard to show that the function
$$
z(s) \doteq f'(s)\frac{\sn_k'(s)}{\sn_k(s)}- f''(s) = \frac{m}{m-1}\frac{V_k(s) v_k'(s)}{v_k^2(s)} - 1.
$$
is non-negative and non-decreasing on $\R^+$. Indeed, from
\begin{equation}\label{derizeta}
z(0)=0, \qquad z'(s) = \frac{m}{v_k(s)} \left[k V_k(s) - \frac{1}{m-1} v_k'(s)z(s)\right]
\end{equation}
we deduce that $z'>0$ when $z<0$, which proves that $z \ge 0$ on $\R^+$. 
%Furthermore, by \eqref{derizeta}, $z'\le 0$ if and only if
%\begin{equation}\label{ine_zle0}
%z(s) \ge k(m-1) \frac{V_k(s)}{v_k'(s)},
%\end{equation}
%Call $p(s)$ the RHS of \eqref{ine_zle0}, and note that, by Proposition \ref{prop_garo} in Appendix 2 below, 
%$$
%p(s) = k(m-1) \frac{V_k(s)}{v_k(s)} \frac{v_k(s)}{v_k'(s)} \qquad \text{is non-decreasing on } \, \R^+, 
%$$
%and $p'(s)>0$ when $k>0$. Let $I$ be a connected component of $\{z-p>0\}$. Since $p$ is strictly increasing, $(z-p)' \le 0$ when $z-p \ge 0$, and therefore necessarily $I = \emptyset$. Hence, $z \le p$ on $\R^+$, which imply $z' \ge 0$ on $\R^+$.\\[0.2cm]
Fix $0<t<s$ regular values for $r$. Integrating \eqref{deltaf} on the smooth compact set $\{t \le r \le s\}$ and using the divergence theorem we deduce
\begin{equation}\label{mono}
\frac{V_k(s)}{v_k(s)} \int_{\Gamma_{\! s}} |\nabla r| - \frac{V_k(t)}{v_k(t)} \int_{\Gamma_t} |\nabla r| \ge \vol\big(\{t \le r \le s\}\big).
\end{equation}
By the definition of $J(s)$ and $\Theta(s)$, and since $J(s) \equiv \bar J(s)$ for regular values, the above inequality rewrites as follows:
$$
V_k(s)\bar J(s) - V_k(t)\bar J(t) \ge V_k(s)\Theta(s) - V_k(t)\Theta(t),
$$
or in other words, 
$$
V_k(s)\big[\bar J(s)-\Theta(s)\big] \ge V_k(t)\big[\bar J(t)-\Theta(t)\big].
$$
Since all the quantities involved are continuous, the above relation extends to all $t,s \in \R^+$, which proves the monotonicity of $V_k[\bar J -\Theta]$. Letting $t \ra 0$ we then deduce that $\bar J(s) \ge \Theta(s)$ on $\R^+$. Next, by using $f \equiv 1$ and $f \equiv |\nabla r|^2$ in Lemma \ref{lem_coarea} and exploiting again Proposition \ref{important!} we get
\begin{equation}\label{inevolume}
\vol\big(\{t \le r \le s\}\big) = \int_t^s \left[\int_{\Gamma_{\!\sigma}} \frac{1}{|\nabla r|}\right]\di \sigma, \qquad \int_{\{0\le r\le s\}}|\nabla r|^2 = \int_0^s \left[\int_{\Gamma_\sigma}|\nabla r|\right]\di \sigma,
\end{equation}
showing that the two quantities in \eqref{monotones} are absolutely continuous. Plugging into \eqref{mono}, letting $t \ra 0$ and using that $z \ge 0$ we deduce 
\begin{equation}\label{fromthis}
\frac{V_k(s)}{v_k(s)} \int_{\Gamma_{\! s}} |\nabla r| \ge \int_0^s \left[\int_{\Gamma_{\!\sigma}} \frac{1}{|\nabla r|}\right]\di \sigma,
\end{equation}
for regular $s$, which together with the trivial inequality $|\nabla r|^{-1} \ge |\nabla r|$ and with \eqref{inevolume} gives
\begin{equation}\label{dueineq}
\begin{array}{l}
\quad \disp V_k(s)\int_{\Gamma_{\! s}} |\nabla r| \ge v_k(s) \int_0^s \left[\int_{\Gamma_{\!\sigma}} |\nabla r|\right]\di \sigma, \\[0.4cm]
\quad \disp V_k(s)\left[ \frac{\di}{\di s} \vol\big( \{ r \le s\}\big)\right] \ge v_k(s) \vol\big(\{r \le s\}\big).
\end{array}
\end{equation}
Integrating the second inequality we obtain the monotonicity of $\Theta(s)$, while integrating the first one and using \eqref{inevolume} we obtain the monotonicity of the second quantity in \eqref{monotones}. To show the monotonicity of $\bar J(s)$, by \eqref{basica} and using the full information coming from \eqref{hessiancomp} we obtain
\begin{equation}\label{deltar}
\frac{\sn_k'(r)}{\sn_k(r)}\big(m-|\nabla r|^2\big) \le \Delta r \le \frac{g'(r)}{g(r)}\big(m-|\nabla r|^2\big).
\end{equation}
In view of the identity \eqref{usefullll}, we consider regular $s>0$, we divide \eqref{deltar} by $|\nabla r|$ and integrate on $\Gamma_s$ to get
%
%Set for convenience 
%\begin{equation}\label{def_Is}
%I(s) \doteq \int_{\Gamma_{\! s}} |\nabla r|, \qquad T(s) \doteq \frac{\int_{\Gamma_{\! s}}|\nabla r|^{-1}}{\int_{\Gamma_{\! s}}|\nabla r|}-1.
%\end{equation}
%Because of Proposition \ref{important!}, applying Lemma \ref{lem_coarea} with the choice $f = \Delta r$ we deduce that
%$$
%\int_{\{t \le r \le s\}} \Delta r = \int_t^s \left[ \int_{\Gamma_{\!\sigma}} \frac{\Delta r}{|\nabla r|} \right] \di \sigma,
%$$
%and by the divergence theorem the left hand-side is $I(s)-I(t)$. Therefore, $I(s)$ is absolutely continuous with derivative
%$$
%\int_{\Gamma_{\! s}} \frac{\Delta r}{|\nabla r|} \in L^1_\loc(\R).
%$$
\begin{equation}\label{diff_vol}
\frac{\sn_k'(s)}{\sn_k(s)}\int_{\Gamma_{\! s}}\frac{m-|\nabla r|^2}{|\nabla r|} \le \big(v_k(s)\bar J(s)\big)' \le \frac{g'(s)}{g(s)}\int_{\Gamma_{\! s}}\frac{m-|\nabla r|^2}{|\nabla r|}.
\end{equation}
Writing $m-|\nabla r|^2 = m(1-|\nabla r|^2) + (m-1)|\nabla r|^2$, setting for convenience 
\begin{equation}\label{def_Is}
v_g(s) = \omega_{m-1}g(s)^{m-1}, \qquad T(s) \doteq \frac{\int_{\Gamma_{\! s}}|\nabla r|^{-1}}{\int_{\Gamma_{\! s}}|\nabla r|}-1,
\end{equation}
rearranging we deduce the two inequalities  
\begin{equation}\label{diff_vol_2}
\begin{array}{rcl}
\big(v_k(s)\bar J(s)\big)' & \ge & \disp v_k'(s)\bar J(s) + m\frac{\sn_k'(s)}{\sn_k(s)}T(s)v_k(s) \bar J(s)\\[0.4cm]
\big(v_k(s)\bar J(s)\big)' & \le & \disp \frac{v_g'(s)}{v_g(s)}v_k(s)\bar J(s) + m\frac{g'(s)}{g(s)}T(s)v_k(s) \bar J(s).
\end{array}
\end{equation}
Expanding the derivative on the left-hand side, we deduce
\begin{equation}\label{diff_vol_3}
\begin{array}{rcl}
\bar J'(s) & \ge & \disp m \frac{\sn_k'(s)}{\sn_k(s)}T(s) \bar J(s), \\[0.4cm]
\disp \left( \frac{v_k(s)}{v_g(s)}\bar J(s)\right)' & \le & \disp m \frac{g'(s)}{g(s)}T(s) \left( \frac{v_k(s)}{v_g(s)}\bar J(s)\right).
\end{array}
\end{equation}
The first inequality together with the non-negativity of $T$ implies the desired $\bar J' \ge 0$, concluding the proof. The second inequality in \eqref{diff_vol_3}, on the other hand, will be useful in awhile.
\end{proof}

\begin{remark}\label{rem_proper}
\emph{The properness of $\varphi$ is essential in the above proof to justify integrations by parts. However, if $\varphi$ is non-proper, at least when $N$ is Cartan-Hadamard with sectional curvature $\bar K \le -k$ the function $\Theta$ is still monotone in an extended sense. In fact, as it has been observed in \cite{Tysk} for $N=\R^{m+1}$, $\Theta(s)= +\infty$ for each $s$ such that $\{r <s\}$ contains a limit point of $\varphi$. Briefly, if $\bar x \in N$ is a limit point with $\bar \rho(\bar x) < s$, choose $\eps>0$ such that $2\eps < s - \bar \rho(\bar x)$, and a diverging sequence $\{x_j\}\subset M$ such that $\varphi(x_j) \ra \bar x$. We can assume that the balls $B_\eps(x_j) \subset M$ are pairwise disjoint. Since $\overline \dist(\varphi(x), \varphi(x_j)) \le \dist(x,x_j)$, we deduce that $\varphi(B_\eps(x_j)) \subset \{r <s\}$ for $j$ large enough, and thus
$$
\vol\big(\{r \le s\}\big) \ge \sum_j \vol(B_\eps(x_j)).
$$
However, using that $\bar K \le -k$ and since $N$ is Cartan-Hadamard, we can apply the intrinsic monotonicity formula (see Proposition \ref{prop_monointri} in Appendix 2 below) with chosen origin $\varphi(x_j)$ to deduce that $\vol(B_\eps(x_j)) \ge V_k(\eps)$ for each $j$, whence $\vol(\{r \le s\}) = +\infty$.
}
\end{remark}

We next investigate conditions equivalent to the finiteness of the density.

\begin{proposition}\label{prop_equivalence}
Suppose that $N$ has a pole and satisfies \eqref{pinchsectio}. Let $\varphi : M^m \ra N^n$ be a proper minimal  immersion. Then, the following properties are equivalent: 
\begin{itemize}
\item[(1)] $\Theta(+\infty)< +\infty$;
\item[(2)] $\bar J(+\infty)<+\infty$.
\end{itemize}
Moreover, both $(1)$ and $(2)$ imply that 
\begin{equation}\label{integrabilitystrange}\tag{$3$}
\frac{\sn_k'(s)}{\sn_k(s)}\left[\frac{\int_{\Gamma_{\! s}}|\nabla r|^{-1}}{\int_{\Gamma_{\! s}}|\nabla r|}-1\right]  \in L^1(\R^+).
\end{equation}
If further $N$ has an integral pinching to $\R^n$ or $\HH^n_k$, then $(1) \Leftrightarrow (2) \Leftrightarrow (3)$.
\end{proposition}

\begin{proof}
We refer to the proof of the previous proposition for notation and formulas.\\
$(2) \Rightarrow (1)$ is obvious since, by the previous proposition, $\bar J(s) \ge \Theta(s)$.\\
$(1) \Rightarrow (2)$. Note that the limit in $(2)$ exists since $\bar J$ is monotone. Suppose by contradiction that $\bar J(+\infty)=+\infty$, let $c>0$ and fix $s_c$ large enough that $\bar J(s) \ge c$ for $s\ge s_c$. From \eqref{inevolume} and \eqref{def_J}, and since $\bar J \equiv J$ a.e., 
$$
\begin{array}{lcl}
\Theta(s) & = & \disp \frac{1}{V_k(s)} \int_0^s \left[ \int_{\Gamma_{\!\sigma}} \frac{1}{|\nabla r|} \right]\di \sigma \ge \frac{1}{V_k(s)}\int_0^s v_k(\sigma)J(\sigma)\di \sigma \\[0.5cm]
& \ge & \disp \frac{1}{V_k(s)}\int^s_{s_c} v_k(\sigma)J(\sigma)\di \sigma \ge c\frac{V_k(s)-V_k(s_c)}{V_k(s)}.
\end{array}
$$
Letting $s \ra +\infty$ we get $\Theta(+\infty) \ge c$, hence $\Theta(+\infty)=+\infty$ by the arbitrariness of $c$, contradicting $(1)$.\\
$(2) \Rightarrow (3)$. Integrating \eqref{diff_vol_3} on $[1,s]$ we obtain
\begin{equation}\label{diff_vol_4}
c_1 \exp\left\{m\int_1^s \frac{\sn_k'(\sigma)}{\sn_k(\sigma)}T(\sigma)\di \sigma\right\} \le \bar J(s) \le c_2 \frac{v_g(s)}{v_k(s)}\exp\left\{m\int_1^s\left[\frac{g'(\sigma)}{g(\sigma)}\right]T(\sigma)\di \sigma \right\}, 
\end{equation}
%\begin{array}{rcl}
%\disp \bar J(s) & \ge & \disp \bar J(1) ; \\[0.4cm]
%\disp \bar J(s) & \le & \disp \frac{v_g(s)}{v_k(s)} \frac{\bar J(1)v_k(1)}{v_g(1)} \exp\left\{m \int_1^s\left[\frac{g'(\sigma)}{g(\sigma)}\right]T(\sigma)\di \sigma \right\},
%\end{array}
%\end{equation}
for some constants $c_1,c_2>0$, where $v_g(s),T(s)$ is as in \eqref{def_Is}. The validity of $(2)$ and the first inequality show that $\sn_k'T/\sn_k \in L^1(+\infty)$, that is, $(3)$ is satisfied.\\
$(3) \Rightarrow (2)$. In our pinching assumptions on $N$, $(ii)$ in Proposition \ref{prop_pinching} gives 
$$
\frac{g'}{g} = \frac{\sn_k'}{\sn_k} + \zeta, \quad \text{with} \quad \zeta \le C \frac{\sn_k'}{\sn_k} \ \text{ on } \, \R^ +, \quad \text{and} \quad g \le C\sn_k \ \text{ on } \, \R^+, 
$$
for some $C>0$. Plugging into \eqref{diff_vol_4} and recalling the definition of $v_g$ we obtain
$$
\disp \bar J(s) \le c_3\exp\left\{c_4\int_1^s\left[\frac{\sn_k'(\sigma)}{\sn_k(\sigma)}\right]T(\sigma)\di \sigma \right\},
$$
for some $c_3,c_4>0$, and $(3) \Rightarrow (2)$ follows by letting $s \ra +\infty$. 
\end{proof}

\begin{remark}
\emph{A version of Propositions \ref{prop_monotonicity} and \ref{prop_equivalence} that covers most of the material presented above has also been independently proved in the very recent \cite{GimenoMarkvosen}, see Theorems 2.1 and 6.1 therein. We mention that their results are stated for more general ambient spaces subjected to specific function-theoretic requirements, and that, in Proposition \ref{prop_equivalence}, it holds in fact $\bar J(+\infty) \equiv \Theta(+\infty)$. For an interesting characterization, when $N=\R^n$, of the limit $\bar J(+\infty)$ in terms of an invariant called the projective volume of $M$ we refer to \cite{tkachev}.  
}
\end{remark}

\section{Proof of Theorem 1}\label{sec_proof}
Let $M^m$ be a minimal properly immersed submanifold in $N^n$, and suppose that $N$ has a pointwise or integral pinching to a space form. Because of the upper bound in \eqref{pinchsectio}, by \cite{Cheung} and \cite{GPB} the bottom of $\sigma(M)$ satisfies
\begin{equation}\label{lowbound_spec}
\inf \sigma(M) \ge \frac{(m-1)^2k}{4}.
\end{equation}
Briefly, the lower bound in \eqref{deltar} implies
$$
\Delta r \ge (m-1) \frac{\sn_k'(r)}{\sn_k(r)} \ge (m-1) \sqrt{k} \qquad \text{on } \, M.
$$
Integrating on a relatively compact, smooth open set $\Omega$ and using the divergence theorem and $|\nabla r| \le 1$, we deduce $\haus^{m-1}(\partial\Omega) \ge (m-1) \sqrt{k}\vol(\Omega)$. The desired \eqref{lowbound_spec} then follows from Cheeger's inequality: 
$$
\inf \sigma(M) \ge \frac{1}{4}\left(\inf_{\Omega \Subset M} \frac{\haus^{m-1}(\partial \Omega)}{\vol(\Omega)}\right)^2 \ge \frac{(m-1)^2k}{4}.
$$ 
To complete the proof of the theorem, since $\sigma (M)$ is closed it is sufficient to show that each $\lambda > (m-1)^2k/4$ lies in $\sigma(M)$. 
%
%\begin{theorem}\cite[Thm 1.1]{elworthywang}
%%
%Let $M^m$ be a complete manifold, and suppose that there exists a proper function $\gamma \in C^0(M)$ which is  $C^2$ on $\{\gamma >R\}$, for some $R>0$. Denote with $A_{t,s} \doteq \{t \le \gamma \le s\}$. For $c>0$, define $\di \mu_c = e^{-c\gamma} \di x$.\\
%If there exists $c \ge 0$ such that
%\begin{equation}
%\begin{array}{ll}
%(1) & \disp \lim_{t \ra +\infty} \limsup_{s \ra +\infty} \frac{1}{\mu_c(A_{t,s})} \int_{A_{t,s}} \Big[ (\Delta \gamma - c)^2 + (|\nabla \gamma|^2-1)^2 \Big] \di \mu_c = 0. \\[0.5cm]
%(2) & \disp \lim_{s \ra +\infty} \max \Big\{ \mu_c(A_{t_0,s}), \mu_c(A_{s,+\infty})^{-1}\Big\} e^{-\eps s} = 0 \qquad \text{for each } \, \eps>0,
%\end{array}
%\end{equation}
%then $\sigma(M) \supset [c^2/4,+\infty)$.
%\end{theorem}
%
%%

Set for convenience $\beta \doteq \sqrt{\lambda - (m-1)^2k/4}$ and, for $0 \le t<s$, let $A_{t,s}$ denote the extrinsic annulus
$$
A_{t,s} \doteq \big\{ x \in M \ : \ r(x) \in [t,s]\big\}.
$$
Define the weighted measure $\di \mu_k \doteq v_k(r)^{-1}\di x$ on $\{r \ge 1\}$. Hereafter, we will always restrict to this set. Consider 
\begin{equation}\label{ODEpsi}
\psi(s) \doteq \frac{e^{i \beta s}}{\sqrt{v_k(s)}}, \qquad \text{which solves} \qquad \psi'' + \psi' \frac{v_k'}{v_k} + \lambda \psi = a(s) \psi, 
\end{equation}
where
\begin{equation}\label{def_as}
a(s) \doteq \frac{(m-1)^2 k}{4} + \frac{1}{4}\left(\frac{v_k'(s)}{v_k(s)}\right)^2 - \frac{1}{2}\frac{v_k''(s)}{v_k(s)} \ra 0 
% = \frac{(m-1)(m-3)}{4}\left[k - \left(\frac{\sn_k'(s)}{\sn_k(s)}\right)^2\right] \ra 0
\end{equation}
as $s \ra +\infty$. For technical reasons, fix $R>1$ large such that $\Theta(R)>0$. Fix $t,s,S$ such that
$$
R+1 < t <  s < S-1, 
$$
and let $\eta \in C^\infty_c(\R)$ be a cut-off function satisfying  
$$
\begin{array}{l}
\disp 0 \le \eta \le 1, \quad \eta\equiv 0 \ \text{ outside of } \,  (t-1,S), \quad \eta \equiv 1 \ \text{ on } \, (t,s), \\[0.2cm]
|\eta'|+ |\eta''| \le C_0 \ \ \text{ on } \, [t-1,s], \qquad |\eta'| + |\eta''| \le \frac{C_0}{S-s} \ \ \text{ on } \, [s,S]
\end{array}
$$
for some absolute constant $C_0$ (the last relation is possible since $S-s \ge 1$). The value $S$ will be chosen later in dependence of $s$. Set $u_{t,s} \doteq \eta(r)\psi(r)\in C^\infty_c(M)$. Then, by \eqref{ODEpsi},
$$
\begin{array}{lcl}
\Delta u_{t,s} + \lambda u_{t,s} & = & \disp (\eta''\psi + 2\eta' \psi' + \eta \psi'')|\nabla r|^2 + (\eta'\psi + \eta \psi') \Delta r + \lambda \eta \psi\\[0.2cm]
%& = & \disp \left(\eta''\psi + 2\eta' \psi' -\frac{v_k'}{v_k} \eta \psi' - \lambda \eta \psi + a\eta\psi\right)|\nabla r|^2 + (\eta'\psi + \eta \psi') \Delta r + \lambda \eta \psi\\[0.3cm]
& = & \disp \left(\eta''\psi + 2\eta' \psi' -\frac{v_k'}{v_k} \eta \psi' - \lambda \eta \psi + a \eta \psi\right)(|\nabla r|^2-1) + a\eta \psi \\[0.3cm]
& & \disp + (\eta' \psi+ \eta \psi')\left(\Delta r - \frac{v_k'}{v_k}\right) + \disp \left(\eta''\psi + 2\eta' \psi' + \eta' \psi \frac{v_k'}{v_k}\right).\\[0.3cm]
\end{array}
$$
Using that there exists an absolute constant $c$ for which $|\psi|+ |\psi'| \le c/\sqrt{v_k}$, the following inequality holds: 
$$
\begin{array}{lcl}
\|\Delta u_{t,s} + \lambda u_{t,s}\|^2_2 & \le & \disp C \left( \int_{A_{t-1,S}}\left[(1-|\nabla r|^2)^2+ \left(\Delta r - \frac{v_k'}{v_k}\right)^2 + a(r)^2\right]\di \mu_k \right.\\[0.5cm]
& & \disp \left. + \frac{\mu_k(A_{s,S})}{(S-s)^2} + \mu_k(A_{t-1,t}) \right), 
\end{array}
$$
for some suitable $C$ depending on $c,C_0$. Since $\|u_{t,s}\|^2_2 \ge \mu_k(A_{t,s})$ and $(1-|\nabla r|^2)^2 \le 1-|\nabla r|^2$, we obtain 
\begin{equation}\label{weelldone!}
\begin{array}{lcl}
\disp\frac{\|\Delta u_{t,s} + \lambda u_{t,s}\|^2_2}{\|u_{t,s}\|_2^2} & \le & \disp C \left( \frac{1}{\mu_k(A_{t,s})}\int_{A_{t-1,S}}\left[1-|\nabla r|^2+ \left(\Delta r - \frac{v_k'}{v_k}\right)^2 + a(r)^2\right]\di \mu_k \right.\\[0.5cm]
& & \disp \left. + \frac{1}{(S-s)^2}\frac{\mu_k(A_{s,S})}{\mu_k(A_{t,s})} + \frac{\mu_k(A_{t-1,t})}{\mu_k(A_{t,s})} \right) 
\end{array}
\end{equation}

Next, using \eqref{hessiancomp},
$$
\begin{array}{lcl}
\Delta r & = & \disp \sum_{j=1}^m \overline{\Hess}(\bar \rho)(e_i,e_i) = \frac{\sn_k'(r)}{\sn_k(r)}(m-|\nabla r|^2) + \mathcal{P}(x) \\[0.3cm]
& = & \disp \frac{v_k'(r)}{v_k(r)} + \frac{\sn_k'(r)}{\sn_k(r)}(1-|\nabla r|^2) + \mathcal{P}(x),
\end{array}
$$
where, by Proposition \ref{prop_pinching},
\begin{equation}\label{esti_T}
\begin{array}{lcl}
0 \le \mathcal{P}(x) & \doteq & \disp \sum_{j=1}^m \overline{\Hess}(\bar \rho)(e_i,e_i) - \frac{\sn_k'(r)}{\sn_k(r)}(m-|\nabla r|^2) \\[0.3cm]
& \le & \disp \left(\frac{g'(r)}{g(r)} - \frac{\sn_k'(r)}{\sn_k(r)}\right)(m-|\nabla r|^2) = \zeta(r)(m-|\nabla r|^2) \le m \zeta(r).
\end{array}
\end{equation}
We thus obtain, on the set $\{r \ge 1\}$,
\begin{equation}\label{stimapadrao_0}
\begin{array}{lcl}
\disp \left(\Delta r - \frac{v_k'}{v_k}\right)^2 + 1-|\nabla r|^2 + a(r)^2 & \le & \disp \left[\frac{\sn_k'(r)}{\sn_k(r)}(1-|\nabla r|^2) + m\zeta(r)\right]^2 \\[0.4cm]
& & \disp + 1-|\nabla r|^2 + a(r)^2 \\[0.3cm]
& \le & \disp C \Big( \zeta(r)^2 + 1-|\nabla r|^2 + a(r)^2\Big)
\end{array}
\end{equation}
for some absolute constant $C$. Note that, in both our pointwise or integral pinching assumptions on $N$, by Proposition \ref{prop_pinching} it holds $\zeta(s) \ra 0$ as $s \ra +\infty$. Set
$$
F(t) \doteq \sup_{\sigma \in [t-1,+\infty)}[a(\sigma)^2+ \zeta(\sigma)^2],
$$
and note that $F(t) \ra 0$ monotonically as $t \ra +\infty$. Integrating \eqref{stimapadrao_0} we get the existence of $C>0$ independent of $s,t$ such that
\begin{equation}\label{simplify!}
\begin{array}{l}
\disp \int_{A_{t-1,S}} \left[\left(\Delta r - \frac{v_k'}{v_k}\right)^2 + 1-|\nabla r|^2 + a(r)^2\right]\di \mu_k \\[0.5cm]
\qquad \qquad \le \disp C \left( F(t)\int_{A_{t-1,S}}\frac{1}{v_k(r)} + \int_{A_{t-1,S}}\frac{1-|\nabla r|^2}{v_k(r)}\right). 
\end{array} 
\end{equation}
%Define 
%$$
%T(s) \doteq \frac{\int_{\Gamma_s}\big[|\nabla r|^{-1}-|\nabla r|\big]}{\int_{\Gamma_s}|\nabla r|}.
%$$
Using the coarea's formula and the transversality lemma, for each $0 \le a<b$
\begin{equation}\label{radialize}
\mu_k(A_{a,b}) = \int_{A_{a,b}} \frac{1}{v_k(r)} = \int_{a}^b J\big[1+T\big], \qquad \int_{A_{a,b}} \frac{1-|\nabla r|^2}{v_k(r)} = \int_{a}^b JT,
\end{equation}
where $J$ and $T$ are defined, respectively, in \eqref{def_J} and \eqref{def_Is}. Summarizing, in view of \eqref{simplify!} and \eqref{radialize} we deduce from \eqref{weelldone!} the following inequalities:
\begin{equation}\label{verygood}
\begin{array}{lcl}
\disp\frac{\|\Delta u_{t,s} + \lambda u_{t,s}\|^2_2}{\|u_{t,s}\|_2^2} & \le & \disp C \left( \frac{1}{\int_t^{s} J\big[1+T\big]}\left[ F(t)\int_{t-1}^SJ\big[1+T\big] + \int_{t-1}^S JT\right]\right.\\[0.5cm]
& & \disp \left. + \frac{\int_{s}^S J\big[1+T\big]}{(S-s)^2\int_{t}^{s} J\big[1+T\big]} + \frac{\int_{t-1}^t J\big[1+T\big]}{\int_t^{s} J\big[1+T\big]} \right) \doteq \mathcal{Q}(t,s).
\end{array}
\end{equation}
If we can guarantee that 
\begin{equation}\label{liminfcond}
\liminf_{t \ra +\infty} \liminf_{s \ra +\infty} \frac{\|\Delta u_{t,s} + \lambda u_{t,s}\|^2_2}{\|u_{t,s}\|_2^2} = 0,
\end{equation}
then we are able to construct a sequence of approximating eigenfunctions for $\lambda$ as follows: fix $\eps>0$. By \eqref{liminfcond} there exists a divergent sequence $\{t_i\}$ such that, for $i \ge i_\eps$, 
$$
\liminf_{s \ra +\infty} \frac{\|\Delta u_{t_i,s} + \lambda u_{t_i,s}\|^2_2}{\|u_{t_i,s}\|_2^2} < \eps/2.
$$
For $i=i_\eps$, pick then a sequence $\{s_j\}$ realizing the liminf. For $j \ge j_\eps(i_\eps,\eps)$ 
\begin{equation}\label{relautisj}
\|\Delta u_{t_i,s_j} + \lambda u_{t_i,s_j}\|^2_2 < \eps \|u_{t_i,s_j}\|_2^2, 
\end{equation}
Writing $u_\eps \doteq u_{t_{i_\eps},s_{j_\eps}}$, by \eqref{relautisj} from the set $\{u_\eps\}$ we can extract a sequence of approximating eigenfunctions for $\lambda$, concluding the proof that $\lambda \in \sigma(M)$. To show \eqref{liminfcond}, by \eqref{verygood} it is enough to prove that 
\begin{equation}\label{desiredliminfs}
\liminf_{t \ra +\infty} \liminf_{s \ra +\infty} \mathcal{Q}(t,s) = 0.
\end{equation}
Suppose, by contradiction, that \eqref{desiredliminfs} were not true. Then, there exists a constant $\delta>0$ such that, for each $t \ge t_\delta$, $\liminf_{s \ra +\infty} \mathcal{Q}(t,s) \ge 2\delta$, and thus for $t \ge t_\delta$ and $s \ge s_\delta(t)$ 
\begin{equation}\label{contraddi}
F(t)\int_{t-1}^SJ\big[1+T\big] + \int_{t-1}^S JT  + \int_{s}^S \frac{J\big[1+T\big]}{(S-s)^2} + \int_{t-1}^t J\big[1+T\big] \ge \delta \int_t^{s} J\big[1+T\big],
\end{equation}
and rearranging
\begin{equation}\label{firstrelation_3}
(F(t)+1)\int_{t-1}^SJ\big[1+T\big] - \int_{t-1}^S J + \int_{s}^S\frac{J\big[1+T\big]}{(S-s)^2} + \int_{t-1}^t J\big[1+T\big] \ge \delta \int_t^{s} J\big[1+T\big].
\end{equation}
We rewrite the above integrals in order to make $\Theta(s)$ appear. Integrating by parts and using again the coarea's formula and the transversality lemma,
\begin{equation}\label{bellaidentita}
\begin{array}{lcl}
\disp \int_a^b J\big[1+T\big] & = & \disp \int_{A_{a,b}} \frac{1}{v_k(r)} = \disp \int_a^b \frac{1}{v_k(\sigma)} \left[\int_{\Gamma_\sigma} \frac{1}{|\nabla r|}\right]\di \sigma = \int_a^b \frac{\big(V_k(\sigma)\Theta(\sigma)\big)'}{v_k(\sigma)}\di \sigma \\[0.5cm]
& = & \disp \frac{V_k(b)}{v_k(b)} \Theta(b)- \frac{V_k(a)}{v_k(a)} \Theta(a)+ \int_{a}^{b} \frac{V_k v_k'}{v_k^2} \Theta.
\end{array}
\end{equation}
To deal with the term containing the integral of $J$ alone in \eqref{firstrelation_3}, we use the inequality $J(s) \ge \Theta(s)$ coming from the monotonicity formulae in Proposition \ref{prop_monotonicity}. This passage is crucial for us to conclude. Inserting \eqref{bellaidentita} and $J \ge \Theta$ into \eqref{firstrelation_3} we get
\begin{equation}\label{firstrelation_5}
\begin{array}{l}
\disp (F(t)+1)\disp \frac{V_k(S)}{v_k(S)} \Theta(S)- (F(t)+1)\frac{V_k(t-1)}{v_k(t-1)} \Theta(t-1) + (F(t)+1)\int_{t-1}^{S} \frac{V_k v_k'}{v_k^2} \Theta \\[0.5cm]
\disp - \int_{t-1}^S \Theta + \disp \frac{1}{(S-s)^2}\left[\frac{V_k(S)}{v_k(S)} \Theta(S)- \frac{V_k(s)}{v_k(s)} \Theta(s)+ \int_{s}^{S} \frac{V_k v_k'}{v_k^2} \Theta\right] + \frac{V_k(t)}{v_k(t)} \Theta(t)\\[0.5cm]
\disp - \frac{V_k(t-1)}{v_k(t-1)} \Theta(t-1) + \int_{t-1}^{t} \frac{V_k v_k'}{v_k^2} \Theta \\[0.5cm]
\qquad \qquad \qquad \ge \quad \delta \disp \frac{V_k(s)}{v_k(s)} \Theta(s)- \delta\frac{V_k(t)}{v_k(t)} \Theta(t)+ \delta\int_{t}^{s} \frac{V_k v_k'}{v_k^2} \Theta.
\end{array}
\end{equation}
The idea to reach the desired contradiction is to prove that, as a consequence of \eqref{firstrelation_5}, 
\begin{equation}\label{inttheyta}
\int_{t-1}^S \Theta
\end{equation}
(hence, $\Theta(S)$) must grow faster as $S \ra +\infty$ than the bound in \eqref{bellissima}. To do so, we need to simplify \eqref{firstrelation_5} in order to find a suitable differential inequality for \eqref{inttheyta}.\\
We first observe that, both for $k>0$ and for $k=0$, there exists an absolute constant $\hat c$ such that $\hat c^{-1} \le V_kv_k'/v_k^2 \le \hat c$ on $[1,+\infty)$. Furthermore, by the monotonicity of $\Theta$, 
\begin{equation}\label{stimasemplice!}
\int_{s}^{S} \frac{V_k v_k'}{v_k^2} \Theta \le \hat c(S-s) \Theta(S).
\end{equation}
Next, we deal with the two terms in the left-hand side of \eqref{firstrelation_5} that involve \eqref{inttheyta}:
$$
\begin{array}{lcl}
\disp (F(t)+1)\int_{t-1}^{S} \frac{V_k v_k'}{v_k^2} \Theta  - \int_{t-1}^S \Theta & = & \disp F(t)\int_{t-1}^{S} \frac{V_k v_k'}{v_k^2} \Theta + \int_{t-1}^S \frac{V_k v_k'-v_k^2}{v_k^2} \Theta \\[0.5cm]
& \le & \disp \hat c F(t)\int_{t-1}^{S}\Theta + \int_{t-1}^S \frac{V_k v_k'-v_k^2}{v_k^2} \Theta.
\end{array}
$$
The key point is the following relation:
\begin{equation}\label{maggica}
\frac{V_k(s) v_k'(s)-v_k(s)^2}{v_k(s)^2} \left\{ \begin{array}{ll}
= -1/m & \quad \text{if } \, k =0; \\[0.2cm]
\ra 0 \ \text{as } \, s \ra +\infty, & \quad \text{if } \, k>0.
\end{array}\right.
\end{equation}
Define
$$
\omega(t) \doteq \sup_{[t-1,+\infty)} \frac{V_k v_k'-v_k^2}{v_k^2}, \qquad \chi(t) \doteq \hat c F(t) + \omega(t).
$$
Again by the monotonicity of $\Theta$,
\begin{equation}\label{cisiamoallafine!!}
\begin{array}{lcl}
\disp (F(t)+1)\int_{t-1}^{S} \frac{V_k v_k'}{v_k^2} \Theta  - \int_{t-1}^S \Theta & \le & \disp \big[\hat c F(t)+\omega(t)\big]\int_{t-1}^{S}\Theta = \chi(t)\int_{t-1}^{S}\Theta \\[0.3cm]
& \le & \disp \chi(t)\Theta(t) + \chi(t)\int_{t}^{S}\Theta.
\end{array}
\end{equation}
For simplicity, hereafter we collect all the terms independent of $s$ in a function that we call $h(t)$, which may vary from line to line. Inserting \eqref{stimasemplice!} and \eqref{cisiamoallafine!!} into \eqref{firstrelation_5} we infer 
\begin{equation}
\begin{array}{l}
\disp \left[\left(F(t)+1+ \frac{1}{(S-s)^2}\right)\frac{V_k(S)}{v_k(S)} + \frac{\hat c}{S-s}\right]\disp \Theta(S) + \chi(t)\int_{t}^{S}\Theta \\[0.5cm]
\disp \ge h(t) + \left(\delta + \frac{1}{(S-s)^2}\right) \frac{V_k(s)}{v_k(s)}\Theta(s) + \delta\hat c^{-1}\int_{t}^{s} \Theta.
\end{array}
\end{equation}
Summing $\delta \hat{c}^{-1}(S-s)\Theta(S)$ to the two sides of the above inequality, using the monotonicity of $\Theta$ and getting rid of the term containing $\Theta(s)$ we obtain 
\begin{equation}\label{firstrelation_7}
\begin{array}{l}
\disp \left[\left(F(t)+1+ \frac{1}{(S-s)^2}\right)\frac{V_k(S)}{v_k(S)} + \frac{\hat c}{S-s} + \delta \hat 
c^{-1}(S-s)\right]\disp \Theta(S) + \chi(t)\int_{t}^{S}\Theta \\[0.4cm]
\disp \ge h(t)  + \delta\hat c^{-1}\int_{t}^{S} \Theta.
\end{array}
\end{equation}
%Using that $V_k/v_k$ is bounded from below on $[1,+\infty)$ and diverges when $k=0$, and taking into account the expression of $\chi(t), \omega(t)$, property \eqref{maggica} and the fact that $F(t) \ra 0$ as $t \ra +\infty$, we can choose $t\ge t_c$ sufficiently large in such a way that the following inequalities are satisfied  for $t \ge t_c$ and $s \ge s_c(t)$:
%$$
%\begin{array}{l}
%\disp \chi(t) < \frac{c \hat c^{-1}}{2}; \\[0.2cm]
%\disp (F(t)+2)\frac{V_k(s)}{v_k(s)} + \hat c +\chi(t) + \frac{c \hat c^{-1}}{2} \le \bar c \frac{V_k(s)}{v_k(s)},
%\end{array}
%$$
%for some suitable $\bar c>0$ independent of $t,s$.
Using \eqref{maggica}, the definition of $\chi(t)$ and the properties of $\omega(t),F(t)$, we can choose $t_\delta$ sufficiently large to guarantee that
\begin{equation}\label{defck}
\delta\hat c^{-1} - \chi(t) \ge c_k \doteq \left\{ \begin{array}{ll} \frac{1}{m} + \frac{\delta \hat c^{-1}}{2} & \quad \text{if } \, k=0, \\[0.2cm] 
\frac{\delta \hat c^{-1}}{2} & \quad \text{if } \, k>0,
\end{array}\right.
\end{equation}
hence
\begin{equation}\label{firstrelation_9}
\disp \left[\left(F(t)+1+ \frac{1}{(S-s)^2}\right)\frac{V_k(S)}{v_k(S)} + \frac{\hat c}{S-s} + \delta \hat 
c^{-1}(S-s)\right]\disp \Theta(S) \ge h(t)  + c_k\int_{t}^{S} \Theta.
\end{equation}
We now specify $S(s)$ depending on whether $k>0$ or $k=0$.\\[0.2cm]
\noindent \emph{The case $k>0$.}\\
We choose $S \doteq s+1$. In view of the fact that $V_k/v_k$ is bounded above on $\R^+$, \eqref{firstrelation_9} becomes
\begin{equation}\label{fine_iperbolica}
\bar c\Theta(s+1) \ge h(t)  + c_k\int_{t}^{s+1} \Theta \ge \frac{c_k}{2}\int_{t}^{s+1} \Theta,
\end{equation}
for some $\bar c$ independent of $t,s$. Note that the last inequality is satisfied provided $s \ge s_\delta(t)$ is chosen to be sufficiently large, since the monotonicity of $\Theta$ implies that $\Theta \not \in L^1(\R^+)$. Integrating and using again the monotonicity of $\Theta$, we get 
$$
(s+1-t)\Theta(s+1) \ge \int_t^{s+1}\Theta \ge \left[\int_{t}^{s_0+1}\Theta\right] \exp\left\{ \frac{c_k}{2\bar c}(s-s_0)\right\},
$$
hence $\Theta(s)$ grows exponentially. Ultimately, this contradicts our assumption \eqref{bellissima}.\\[0.2cm]
\noindent \emph{The case $k=0$.}\\ 
We choose $S \doteq s + \sqrt{s}$. Since $V_k(S)/v_k(S) = S/m$, from \eqref{firstrelation_9} we infer 
\begin{equation}\label{firstrelation_25}
\disp \left[\left(F(t)+1+ \frac{1}{s}\right)\frac{S}{m} + \frac{\hat c}{\sqrt{s}} + \delta \hat 
c^{-1}\sqrt{s}\right]\disp \Theta(S) \ge h(t)  + c_k\int_{t}^{S} \Theta.
\end{equation}
Using the expression of $c_k$ and the fact that $F(t) \ra 0$, up to choosing $t_\delta$ and then $s_\delta(t)$ large enough we can ensure the validity of the following inequality:
$$
\left[\left(F(t)+1+ \frac{1}{s}\right)\frac{S}{m} + \frac{\hat c}{\sqrt{s}} + \delta \hat 
c^{-1}\sqrt{s}\right] < \left[\frac{1}{m} + \frac{\delta \hat c^{-1}}{4}\right]S = \left[c_k - \frac{\delta \hat c^{-1}}{4}\right]S
$$
for $t \ge t_\delta$ and $s \ge s_\delta(t)$. Plugging into \eqref{firstrelation_9}, and using that $\Theta \not \in L^1(\R^+)$,  
$$
S\Theta(S) \ge h(t) + \frac{c_k}{c_k- \delta \hat{c}^{-1}/4} \int_t^S\Theta \ge (1+\eps)\int_t^S\Theta,  
$$
for a suitable $\eps>0$ independent of $t,S$, and provided that $S \ge s_\delta(t)$ is large enough. Integrating and using again the monotonicity of $\Theta$, 
$$
S\Theta(S) \ge (S-t)\Theta(S) \ge \int_t^S\Theta \ge \left[\int_t^{S_0}\Theta\right]\left(\frac{S}{S_0}\right)^{1+\eps}, 
$$
hence $\Theta(S)$ grows polynomially at least with power $\eps$, contradicting \eqref{bellissima}.\\
Concluding, both for $k>0$ and for $k=0$ assuming \eqref{contraddi} leads to a contradiction with our assumption \eqref{bellissima}, hence \eqref{liminfcond} holds, as required.

\section{Proof of Theorem 2}
We first show that $\varphi$ is proper and that $M$ is diffeomorphic to the interior of a compact manifold with boundary. Both the properties are consequence of the following lemma due to \cite{BC}, which improves on \cite{Anderson_prep}, \cite{filho}, \cite{castillon}, \cite{BJM}.

\begin{lemma}\label{lem_proper}
Let $\varphi : M^m \ra N^n$ be an immersed submanifold into an ambient manifold $N$ with a pole and suppose that $N$ satisfies \eqref{pinchsectio} for some $k \ge 0$. 
Denote by $B_s = \{x \in M;\; \rho(x) \le s\}$  the intrinsic ball on $M$. Assume that
\begin{equation}\label{assu_II}
\begin{array}{rll}
(i) &\quad \disp \limsup_{s \ra +\infty} s\|\II\|_{L^\infty(\partial B_s)} < 1 & \quad \text{if } \, k = 0 \text{ in \eqref{pinchsectio}, or} \\[0.4cm] 
(ii) & \quad \disp \limsup_{s \ra +\infty} \|\II\|_{L^\infty(\partial B_s)} < \sqrt{k} & \quad \text{if } \, k > 0 \text{ in \eqref{pinchsectio}}.
\end{array}
\end{equation}

Then, $\varphi$ is proper and there exists $R>0$ such that $|\nabla r|>0$ on $\{r \ge R\}$, where $r$ is the extrinsic distance function. Consequently, the flow
\begin{equation}\label{def_flow}
\Phi : \R^+ \times \{r=R\} \ra \{r \ge R\}, \qquad \frac{\di}{\di s}\Phi_s(x) = \frac{\nabla r}{|\nabla r|^2}\big(\Phi_s(x)\big)
\end{equation}
is well defined, and $M$ is diffeomorphic to the interior of a compact manifold with boundary.
\end{lemma}

The properness of $\varphi$ enables us to apply Proposition \ref{prop_equivalence}. Therefore, to show that $\Theta(+\infty)<+\infty$ it is enough to check that 
\begin{equation}\label{rapido}
\frac{\sn_k'(s)}{\sn_k(s)} \frac{\int_{\Gamma_s} \big[|\nabla r|^{-1} - |\nabla r|\big]}{ \int_{\Gamma_s}|\nabla r|} \in L^1(+\infty).
\end{equation}

To achieve \eqref{rapido}, we need to bound from above the rate of approaching of $|\nabla r|$ to $1$ along the flow $\Phi$ in Lemma \ref{lem_proper}. We begin with the following

\begin{lemma}\label{lem_computations}
Suppose that $N$ has a pole and radial sectional curvature satisfying \eqref{pinchsectio}, and that $\varphi: M^m \ra N^n$ is a proper minimal immersion such that $|\nabla r|>0$ outside of some compact set $\{r \le R\}$. Let $\Phi$ denote the flow of $\nabla r/|\nabla r|^2$ as in \eqref{def_flow} and let $\gamma : [R, +\infty) \ra M$ be a flow line starting from some $x_0 \in \{r=R\}$. Then, along $\gamma$, 
\begin{equation}\label{belleequation}
\disp \frac{\di}{\di s}\big( \sn_k(r)\sqrt{1-|\nabla r|^2}\big) \le \disp \sn_k(r) |\II(\gamma(s))|
\end{equation}
%In particular, for each $s>t \ge R$,
%\begin{equation}\label{upperbound}
%1-|\nabla r(\gamma(s))|^2 \le 2\left[\frac{\sn_k(t)}{\sn_k(s)}\right]^2 + \frac{2}{\sn_k(s)^2}\left[ \int_t^s |\II(\gamma(\sigma))|\sn_k(\sigma)\di \sigma \right]^2.
%\end{equation}
\end{lemma}

\begin{proof}
Observe that $r(\gamma(s))=s-R$. By the chain rule and the Hessian comparison theorem \ref{hessiancomp}, 
$$
\begin{array}{lcl}
\disp \frac{\di}{\di s}|\nabla r|^2 &= & \disp 2\Hess r(\nabla r, \dot \gamma) = \frac{2}{|\nabla r|^2} \Hess r(\nabla r, \nabla r) \\[0.4cm]
& = & \disp \frac{2}{|\nabla r|^2} \overline{\Hess}(\bar \rho) \big(\di \varphi(\nabla r), \di \varphi(\nabla r) \big) + \frac{2}{|\nabla r|^2} \big( \bar \nabla \bar \rho, \II(\nabla r, \nabla r)\big) \\[0.4cm]
& \ge & \disp 2\frac{\sn'_k(r)}{\sn_k(r)}(1-|\nabla r|^2) - 2|\bar \nabla^\perp \bar \rho||\II|, 
\end{array}
$$
where $\bar \nabla^\perp \bar \rho$ is the component of $\bar\rho$ perpendicular to $\di \varphi(TM)$ and $|\bar \nabla^\perp \rho| = \sqrt{1-|\nabla r|^2}$. Then,
$$
\disp \frac{\di}{\di s}|\nabla r|^2 \ge \disp 2\frac{\sn'_k(r)}{\sn_k(r)}(1-|\nabla r|^2) - 2|\II|\sqrt{1-|\nabla r|^2}.  
$$
Multiplying by $\sn_k^2(r)$ gives
$$
\disp \frac{\di}{\di s}\big(\sn^2_k(r)(1-|\nabla r|^2)\big) \le 2\sn_k^2(r)|\II|\sqrt{1-|\nabla r|^2},   
$$
which implies \eqref{belleequation}. 
%Integrating on $[s,t]$ we deduce that 
%\begin{equation}
%\begin{array}{lcl}
%\disp \sn_k(s) \sqrt{1-|\nabla r(\gamma(s))|^2} & \le & \disp \sn_k(t) \sqrt{1-|\nabla r(\gamma(t))|^2} + \int_t^s |\II(\gamma(\sigma))|\sn_k(\sigma)\di \sigma \\[0.3cm]
%& \le & \disp \sn_k(t) + \int_t^s |\II(\gamma(\sigma))|\sn_k(\sigma)\di \sigma,
%\end{array}
%\end{equation}
%and \eqref{upperbound} follows by squaring and using the inequality $(a+b)^2 \le 2a^2+2b^2$.
\end{proof}

The above lemma relates the behaviour of $|\nabla r|$ to that of the second fundamental form. The next result makes this relation explicit in the two cases considered in Theorem \ref{teo_finitedens}. 
\begin{proposition}\label{prop_inte}
In the assumptions of the above proposition, suppose further that either
\begin{equation}\label{assu_IIdecay_1}
\begin{array}{rll}
(i) & \quad \disp \|\II\|_{L^\infty(\partial B_s)} \le \frac{C}{s \log^{\alpha/2} s} & \quad \text{if } \, k = 0 \text{ in \eqref{pinchsectio}, or} \\[0.4cm] 
(ii) & \quad \disp \|\II\|_{L^\infty(\partial B_s)} \le \frac{C}{\sqrt{s} \log^{\alpha/2} s} & \quad \text{if } \, k > 0 \text{ in \eqref{pinchsectio}.}
\end{array}
\end{equation}
for $s \ge 1$ and some constants $C>0$ and $\alpha>0$. Here, $\partial B_s$ is the boundary of the intrinsic ball $B_s(o)$. Then, $|\nabla r|(\gamma(s)) \ra 1$ as $s$ diverges, and if $s>2R$ and $R$ is sufficiently large,
\begin{equation}\label{integrability}
\begin{array}{ll}
\disp \text{in the case $(i)$,} & \disp \qquad 1-|\nabla r(\gamma(s))|^2 \le \frac{\hat C}{\log^\alpha s} \\[0.4cm]
\disp \text{in the case $(ii)$,} & \disp \qquad 1-|\nabla r(\gamma(s))|^2 \le \frac{\hat C}{s\log^\alpha s}
\end{array}
\end{equation}
for some constant $\hat C$ depending on $R$.
\end{proposition}

\begin{proof}
We begin by observing that, in \eqref{assu_IIdecay_1}, $\partial B_s$ can be replaced by $\Gamma_s$. Indeed, since $r(x) \le r(o) + \rho(x)$, we can choose $R$ large enough depending on $r(o),\alpha$ in such a way that, for instance in $(i)$,
$$
|\II(x)| \le \frac{C}{\rho(x)\log^{\alpha/2}\rho(x)} \le \frac{C_1}{r(x)\log^{\alpha/2}r(x)}
$$
for some absolute $C_1$ and for each $r\ge R$. Thus, from $(i)$ and $(ii)$ we infer the bounds
\begin{equation}\label{better}
\|\II\|_{L^\infty(\Gamma_s)} \le \frac{C_1}{s \log^{\alpha/2} s} \quad \text{for } \, (i), \qquad \|\II\|_{L^\infty(\Gamma_s)} \le \frac{C_1}{\sqrt{s} \log^{\alpha/2} s} \quad \text{for } \, (ii).
\end{equation}
Because of \eqref{better}, up to enlarging $R$ further there exists a uniform constant $C_2>0$ such that, on $[R, +\infty)$, 
\begin{equation}\label{realnalaysis}
\sn_k(s)|\II(\gamma(s))| \le \left\{ \begin{array}{ll} \disp \frac{C_1}{\log^{\alpha/2}s} \le C_2 \frac{\di}{\di s}\left(\frac{s}{\log^{\alpha/2}s}\right) & \quad \text{if } \, k=0;\\[0.5cm]
\disp \frac{C_1 \sn_k(s)}{\sqrt{s} \log^{\alpha/2}s} \le C_2 \frac{\di}{\di s}\left(\frac{\sn_k(s)}{\sqrt{s} \log^{\alpha/2}s}\right) & \quad \text{if } \, k>0.
\end{array}\right.
\end{equation}
Integrating on $[R,s]$ and using \eqref{belleequation} we get
$$
\sqrt{1-|\nabla r(\gamma(s))|^2} \le \left\{ \begin{array}{ll}
\disp \frac{C_3(R)}{s} + \frac{C_4}{\log^{\alpha/2} s} \le \frac{C_5}{\log^{\alpha/2}s} & \quad \text{if } \, k=0, \\[0.5cm]
\disp \frac{C_3(R)}{\sn_k(s)} + \frac{C_4}{\sqrt{s}\log^{\alpha/2} s} \le \frac{C_5}{\sqrt{s}\log^{\alpha/2}s} & \quad \text{if } \, k>0,
\end{array}\right.
$$
for some absolute constants $C_4,C_5>0$ and if $s > 2R$ and $R$ is large enough. The desired \eqref{integrability} follows by taking squares.

\end{proof}

We are now ready to conclude the proof of Theorem \ref{teo_finitedens} by showing that $M$ has finite density or, equivalently, that \eqref{rapido} holds.\par

%\begin{proposition}\label{prop_finitedens}
%Suppose that $N$ has a pole and an integral pinching towards $\R^n$ or $\HH^n_k$. Let $\varphi : M^m \ra N^n$ be a minimal submanifold whose second fundamental form $\II$ decays as in \eqref{assu_IIdecay_1}, for some $\alpha>1$. Then, $M$ is proper, the extrinsic distance $r$ has no critical points outside some compact set, and $\Theta(+\infty)< +\infty$.
%\end{proposition}
%
Let $\eta(s)$ be either
\begin{equation}\label{def_eta}
\frac{1}{\log^\alpha s} \ \, \text{ when } k=0, \text{ or } \, \frac{1}{s\log^\alpha s} \ \, \text{ when } k>0,
\end{equation}
where $\alpha>1$ and $C$ is a large constant. In our assumptions, we can apply Lemma \ref{lem_computations} and Proposition \ref{prop_inte} to deduce, according to \eqref{integrability}, that, for large enough $R$,
$$
1-|\nabla r(\gamma(s))|^2 \le C\eta(s) \qquad \text{on } \, (R, +\infty), 
$$
where $\gamma(s)$ is a flow curve of $\Phi$ in \eqref{def_flow} and $C=C(R)$ is a large constant. In particular, $|\nabla r(\gamma(s))| \ra 1$ as $s \ra +\infty$. We therefore deduce the existence of a constant $C_2(R)>0$ such that 
$$
\frac{\sn_k'(s)}{\sn_k(s)} \frac{\int_{\Gamma_s} \big[|\nabla r|^{-1} - |\nabla r|\big]}{ \int_{\Gamma_s}|\nabla r|} \le C\frac{\sn_k'(s)}{\sn_k(s)}\eta(s) \frac{\int_{\Gamma_s} |\nabla r|^{-1}}{\int_{\Gamma_s}|\nabla r|} \le C_2\frac{\sn_k'(s)}{\sn_k(s)} \eta(s).
$$
In both our cases $k=0$ and $k>0$, since $\alpha >1$ it is immediate to check that $\sn_k'\eta/\sn_k \in L^1(+\infty)$, proving \eqref{rapido}.

\section*{Appendix 1: finite total curvature solutions of Plateau's problem}
In this appendix, we show that (smooth) solutions of Plateau's problem at infinity $M^m \ra \HH^n$ have finite total curvature whenever $M$ is a hypersurface and the boundary datum $\Sigma \subset \partial_\infty \HH^n$ is sufficiently regular. Consider the Poincar\'e model of $\HH^n$, and let $M \ra \HH^n$ be a proper minimal submanifold. We say that $M$ is $C^{k,\alpha}$ up to $\partial_\infty \HH^n$ if its closure $\overline{M}$ in the topology of the closed unit ball $\overline{\HH^n} = \HH^n \cup \partial_\infty \HH^n$ is a $C^{k,\alpha}$-manifold with boundary. We begin with a lemma, whose proof have been suggested to the second author by L. Mazet.

\begin{lemma}\label{prop_mazet}
Let $\varphi : M^m \ra \HH^n$ be a proper minimal submanifold. If $M$ is of class $C^2$ up to $\partial_\infty \HH^n$, then $M$ has finite total curvature.
\end{lemma}

\begin{proof}
The Euclidean metric $\overline{\metric}$ is related to the Poincar\'e metric $\metric$ by the formula
$$
\overline{\metric} = \lambda^2 \metric, \qquad \text{with} \quad \lambda = \frac{1-|x|^2}{2}.
$$
Given a proper, minimal submanifold $\varphi : (M^m,g) \ra (\HH^n, \metric)$, we associate the isometric immersion $\bar \varphi : (M, (\lambda^2 \circ \varphi) g) \ra (\HH^{n}, \overline{\metric})$, $\bar \varphi(x) \doteq \varphi(x)$. 
Fix a local Darboux frame $\{e_i, e_\alpha\}$ on $(M,g)$ for $\varphi$, with $\{e_i\}$ tangent to $M$ and $\{e_\alpha\}$ in the normal bundle, and let $\bar e_i = e_i/\lambda$, $\bar e_\alpha = e_\alpha/\lambda$ be the corresponding Darboux frame on $(M, \lambda^2g)$ for $\bar \varphi$. Let $\di V$ and $\di \bar V = \lambda^{m} \di V$ be the volume forms of $(M,g)$ and $(M, \lambda^2g)$, and denote with $h^\alpha_{ij}$ and $\bar h^\alpha_{ij}$ the coefficients of the second fundamental forms of $\varphi$ and $\bar \varphi$, respectively. A standard computation shows that 
$$
\bar h^\alpha_{ij} = \frac{1}{\lambda} h^\alpha_{ij} - \frac{\lambda_\alpha}{\lambda} \delta_{ij},
$$
where $\lambda_\alpha = e_\alpha(\lambda)$. Evaluating the norms of $\II$ and $\bar \II$, since $h^\alpha_{ij}$ is trace-free by minimality we obtain 
$$
|\bar\II|^2 = \lambda^{-2}|\II|^2 + m |\nabla^\perp \log \lambda|^2 \ge \lambda^{-2}|\II|^2, 
$$
and thus $|\bar \II|^m \di \bar V \ge |\II|^m \di V$. Integrating on $M$ it holds
$$
\int_M |\II|^m \di V \le \int_M |\bar\II|^m \di \bar V.
$$ 
However, the last integral is finite since $M$ is $C^2$ up to $\partial_\infty \HH^n$, and thus $\varphi$ has finite total curvature.
\end{proof}
In view of Lemma \ref{prop_mazet}, we briefly survey on some boundary regularity results for solutions of Plateau's problem. To the best of our knowledge, we just found regularity results for hypersurfaces. Let $M^m \ra \HH^{m+1}$ be a solution of Plateau's problem for a compact, $(m-1)$-dimensional submanifold $\Sigma^{m-1} \subset \partial_\infty \HH^{m+1}$. Then, a classical result of  Hardt and Lin \cite{HardtLin} states that if $\Sigma^{m-1} \hookrightarrow \partial_\infty \HH^{m+1}$ is properly embedded and $C^{1,\alpha}$, with $0 \le \alpha \le 1$, near $\Sigma$ each solution $M^m \ra \HH^{m+1}$ of Plateau's problem is a finite collection of $C^{1,\alpha}$-manifolds with boundary, which are disjoint except at the boundary. Therefore, near $\Sigma$, $M$ can locally be described as a graph, and the higher regularity theory in \cite{Lin} \cite{Lin2},  \cite{tonegawa2}, \cite{tonegawa} applies to give the following: if $\Sigma$ is $C^{j,\alpha}$, then $M$ is $C^{j,\alpha}$ up to $\partial_\infty \HH^{m+1}$ whenever 
\begin{itemize}
\item[-] $1 \le j \le m-1$ and $0 \le \alpha \le 1$, or
\item[-] $j=m$ and $0 < \alpha < 1$, or
\item[-] $j \ge m+1$ and $0<\alpha<1$ (if $m$ is odd, under a further condition on $\Sigma$).
\end{itemize}
The reader can consult the statement and references in \cite{Lin2}. In particular, because of Lemma \ref{prop_mazet}, if $\Sigma$ is $C^{2,\alpha}$ for some $0<\alpha<1$ then $M$ has finite total curvature (provided that it is smooth). 

\section*{Appendix 2: the intrinsic monotonicity formula}
We conclude by recalling an intrinsic version of the monotonicity formula. To state it, we premit the following observation due to H. Donnelly and N. Garofalo,  Proposition 3.6 in \cite{donnellygarofalo}.

\begin{proposition}\label{prop_garo}
For $k \ge 0$, the function 
\begin{equation}\label{monoto_vk}
\frac{V_k(s)}{v_k(s)} \qquad \text{is non-decreasing on } \, \R^+.
\end{equation}
\end{proposition}
\begin{proof}
The ratio $v_k'/v_k$ is monotone decreasing by the very definition of $v_k$. Then, since $v_k'>0$, the desired monotonicity follows from a lemma at p. 42 of \cite{cheegergromovtaylor}.
\end{proof}

\begin{proposition}[The intrinsic monotonicity formula]\label{prop_monointri}
Suppose that $N$ has a pole $\bar o$ and satisfies \eqref{pinchsectio}, and let $\varphi : M^m \ra N^n$ be a complete, minimal immersion. Suppose that $\bar o \in \varphi(M)$, and choose $o \in M$ be such that $\varphi(o) = \bar o$. Then, denoting with $\rho$ the intrinsic distance function from $o$ and with $B_s = \{ \rho \le s\}$, 
\begin{equation}\label{intrinsicquotient}
\frac{\vol(B_s)}{V_k(s)}
\end{equation}
is monotone non-decreasing on $\R^+$.
\end{proposition}
\begin{proof}
We refer to Proposition \ref{prop_monotonicity} for definitions and computations. We know that the function $\psi= f\circ r$, with $f$ as in \eqref{def_f}, solves $\Delta \psi \ge 1$ on $M$. Integrating on $B_s$ and using the definition of $\psi$ we obtain
$$
\vol(B_s) \le \int_{B_s}\Delta \psi = \int_{\partial B_s}\langle \nabla \psi, \nabla \rho \rangle \le \int_{\partial B_s} \frac{V_k(r)}{v_k(r)}.
$$
Next, since $\bar o = \varphi(o)$, it holds $r(x) \le \rho(x)$ on $M$. Using then Proposition \ref{prop_garo},  we deduce
$$
\vol(B_s) \le \frac{V_k(s)}{v_k(s)} \vol(\partial B_s).
$$
Integrating we obtain the monotonicity of the desired \eqref{intrinsicquotient}.
\end{proof}

\vspace{0.5cm}

\noindent \textbf{Acknowledgements}.\\
The second author is supported by the grant PRONEX - N\'ucleo de An\'alise Geom\'etrica e Aplicac\~oes Processo nº PR2-0054-00009.01.00/11. The third author is partially supported by CNPq. The second author would like to thank L. Hauswirth and L. Mazet for an interesting discussion on finite total curvature submanifolds of $\HH^n$, A. Figalli and G.P. Bessa for a hint, P. Castillon for a bibliographical suggestion and V. Gimeno for pleasant conversations and various comments that lead to several improvements after we posted a first version of the paper on arXiv.



\end{document}